\documentclass[a4paper,11pt]{amsart}
\usepackage{amsmath,amsthm,amssymb}
\usepackage[mathscr]{eucal}
 \usepackage{cite}
\usepackage{upgreek}
\usepackage[bookmarks=false]{hyperref}
\usepackage{enumerate}
\usepackage{amsmath}
\usepackage{mathrsfs}
\usepackage{blindtext}
\usepackage{scrextend}
\usepackage{enumitem}
\usepackage{bm} 
\addtokomafont{labelinglabel}{\sffamily}
\usepackage{color}
\usepackage[at]{easylist}

\numberwithin{equation}{section}

\theoremstyle{plain}

\newtheorem{theorem}{Theorem}[section]
\newtheorem{claim}[theorem]{Claim}
\newtheorem{lemma}[theorem]{Lemma}
\newtheorem{proposition}[theorem]{Proposition}
\newtheorem{corollary}[theorem]{Corollary}
\theoremstyle{definition}
\newtheorem{definition}[theorem]{Definition}
\newtheorem*{definition*}{Definition}
\newtheorem{example}[theorem]{Example}
\newtheorem{notation}[theorem]{Notation}

\newtheorem{remark}[theorem]{Remark}

\newtheorem{conjecture}[theorem]{Conjecture}
\newtheorem{op}[theorem]{Open Problem}

\newcommand{\ra}{\rightarrow}

\begin{document}

\title[Tensor product indecomposability results for existentially closed factors]
{Tensor product indecomposability results for existentially closed factors}

\author[I. Chifan]{Ionu\c t Chifan}
\address{14 MacLean Hall, Department of Mathematics, The University of Iowa, IA, 52242, U.S.A.}\email{ionut-chifan@uiowa.edu}

\author[D. Drimbe]{Daniel Drimbe}
\address{Department of Mathematics, KU Leuven, Celestijnenlaan 200b, B-3001 Leuven, Belgium}\email{daniel.drimbe@kuleuven.be}

\author[A. Ioana]{Adrian Ioana}
\address{Department of Mathematics, University of California San Diego, 9500 Gilman Drive, La Jolla, CA 92093, USA}\email{aioana@ucsd.edu}

{\thanks{I.C. was partially supported by NSF Grants DMS-2154637 and FRG-DMS-1854194; D.D. was supported by the postdoctoral fellowship fundamental research 12T5221N of the Research Foundation Flanders; A.I. was partially supported by NSF FRG Grant \#1854074.}}
\begin{abstract} In the first part of the paper we survey several results from Popa's deformation/rigidity theory on the classification of tensor product decompositions of large natural classes of II$_1$ factors.  Using a m\'elange of techniques from deformation/rigidity theory, model theory, and the recent works \cite{CIOS21,CDI22} we highlight an uncountable family of  existentially closed  II$_1$ factors $M$ which do not admit tensor product decompositions  $M= P\overline \otimes Q$ into diffuse factors where  $Q$ is full. In the last section we discuss several open problems regarding the structural theory of existentially closed factors. 
 \end{abstract}

\maketitle

\section{Introduction}

A {\it von Neumann algebra} is an algebra of bounded linear operators on a Hilbert space which is closed under the adjoint operation and in the weak operator topology. 
A central theme in operator algebras is the study of tensor product decompositions of 
II$_1$ {\it factors}: indecomposable infinite dimensional von Neumann algebras which admit a trace.  
A II$_1$ factor which does not admit a tensor product decomposition into diffuse factors is called \emph{prime}. Using the notion of $*$-orthogonal von Neumann algebras, Popa showed in \cite{Po83} that the II$_1$ factor $L(\mathbb F_S)$ arising from the free group $\mathbb F_S$ with uncountably many generators $S$ is prime. Using Voiculescu's free probability theory, Ge  obtained the first examples of separable prime II$_1$ factors by showing that 
the free group factors $L(\mathbb F_n)$, with $n\ge 2$, have this property \cite{Ge96}.  

These results have been since generalized and strengthened in several ways.
Ozawa  used strong C$^*$-algebraic techniques  to prove that for any icc hyperbolic group $\Gamma$, the II$_1$ factor $ L(\Gamma)$ is \emph{solid}, meaning that  the relative commutant of any diffuse subalgebra is amenable  \cite{Oz03}. By developing a new approach approach on closable derivations, Peterson showed primeness of $L(\Gamma)$, whenever $\Gamma$ is any nonamenable icc group with positive first Betti number \cite{Pe06}. Using the framework of his powerful deformation/rigidity theory, Popa gave a new proof of the solidity of $L(\mathbb F_n)$ \cite{Po06b}. 
Subsequently, an intense activity in the study of tensor product decompositions of II$_1$ factors  led to a plethora of new examples of prime II$_1$ factors arising from various classes of countable groups and their measure preserving actions \cite{Oz04, Po06a, CI08, CH08, Va10b, Bo12, HV12, DI12, CKP14, Ho15, DHI16, CdSS17, Dr19a, IM19}.

In most of these results, the groups $\Gamma$ for which $L(\Gamma)$ was proven to be prime either satisfy an algebraic assumption (e.g.\ $\Gamma$ is an amalgamated free product group or a
wreath product group) or have some ``rank one" properties 
(e.g., $\Gamma$ is a lattice in a rank one simple Lie group or admits a certain unbounded quasi-cocycle). On the other hand, the primeness problem for higher rank lattices
is wide open. A general conjecture predicts that any icc irreducible lattice $\Gamma$ in a product $G_1\times\dots\times G_n$ of connected non-compact simple real Lie groups with finite center gives rise to a prime II$_1$ factor  \cite[Problem V]{Io17}. This is a von Neumann algebraic counterpart of the fact, implied by Margulis' normal subgroup theorem (see \cite[Theorem 8.1.1]{Zi84}), that $\Gamma$ does not admit any direct product decomposition into infinite groups.  Using methods from Popa's deformation/rigidity theory, this conjecture has been confirmed when $G_1,\dots, G_n$ have rank one \cite{DHI16}.
On the other hand, when at least one of the groups $G_1,\dots, G_n$ has rank greater than one, the conjecture is notoriously hard. Some positive evidence in this case has been obtained in \cite{IM19}. Using results from \cite{Io08,BIP18}, it was shown in \cite{IM19} that any profinite free ergodic probability measure preserving action of $\Gamma$ gives rise to a prime II$_1$ factor.

The first main goal of the paper is to discuss 
Popa’s deformation/rigidity theory. This is a remarkable framework that, starting in the early 2000s, has led to unprecedented
progress in the theory of von Neumann algebras (see the survey papers \cite{Po07,Va10,
Io12, Io17}). At the heart of Popa’s theory is the innovative idea of using
deformations of II$_1$ factors to locate rigid subalgebras. We briefly explain this principle in Section \ref{section.popa.deformation.rigidity} and, in addition, we present Popa’s proof \cite{Po06b} of Ozawa’s solidity result \cite{Oz03} of the free group factors.

The second main goal of this paper is to study tensor product decompositions for certain natural families of II$_1$ factors arising from model theory \cite{GHS13,FGHS16}. 
To motivate our results, we recall that 
existentially closed groups are simple \cite{HS88} and as such do not admit any nontrivial direct or semi-direct product decompositions. Since existentially closed factors are von Neumann algebraic analogues of existentially closed groups, it is reasonable to expect they share similar indecomposability properties.
It is known that existentially closed factors $M$ are {\it McDuff}, i.e. $M\cong M\bar\otimes R$ where $R$ is the hyperfinite II$_1$ factor, \cite{GHS13}, but {\it not strongly McDuff}\footnote{This terminology was introduced by Popa in \cite{Po09}.}, i.e. $M$ does not admit a tensor product decomposition $B\bar\otimes R$ with $B$ a full II$_1$ factor and $R$ the hyperfinite II$_1$ factor,  \cite{AGKE20}.
However, besides this result, little is known regarding possible tensor decompositions of existentially closed factors.

In this paper we make progress on this problem by  obtaining in Sections \ref{Section.ec.factors} and \ref{Section.ec.factors.indecomposability}  a series of tensor indecomposability results for existentially closed factors. First,
by building upon \cite{CIOS21,CDI22} (see Section \ref{section.wreath}) and using Popa's deformation/rigidity theory we provide in Theorem \ref{etprime} an uncountable family of existentially closed II$_1$ factors $M$ that do not admit tensor product decompositions  $M=P\overline \otimes Q$ into diffuse factors with $Q$ full. 
In fact, we believe that this property holds for any existentially closed factor $M$ and we provide several instances when it holds under certain additional assumptions on $P$ or $Q$. For instance, we establish the property whenever $Q$ is a group von Neumann algebra of a non-inner amenable group,  see Theorem \ref{theorem.inneramenable}. Finally, we conclude our paper by proposing in Section \ref{section.open}  several open problems regarding the structure of existentially closed factors beyond tensor indecomposability.

\section{Preliminaries on von Neumann Algebras}

\subsection{Terminology} A {\it tracial von Neumann algebra} is a pair $(M,\tau)$ consisting of a von Neumann algebra $M$ and a faithful normal tracial state $\tau:M\rightarrow\mathbb C$. This induces a norm on $M$ by the formula $\|x\|_2=\tau(x^*x)^{1/2},$ for all $x\in M$.
We denote by $\mathscr U(M)$ the group of {\it unitary} elements of $M$, by $\mathcal Z(M)$ the {\it center} of $M$ and by $\text{Aut}(M)$ the group of $\tau$-preserving automorphisms of $M$. 

For $u\in\mathscr U(M)$, we denote by $\text{Ad}(u)\in\text{Aut}(M)$ the inner automorphism of $M$ given by $\text{Ad}(u)(x)=uxu^*$. The group of all inner automorphisms of $M$ is denoted by $\text{Inn}(M)$. For  a set $I$, we denote by $(M^I,\tau)$ the tensor product tracial von Neumann algebra $\overline{\otimes}_{i\in I}(M, \tau)$. For $J\subset I$, we view $M^J\subset M^I$, in the natural way. For $i\in I$, we write $M^i$ instead of $M^{\{i\}}$.

All inclusions $P\subset M$ of von Neumann algebras are assumed unital unless otherwise stated. We denote by $E_{P}:M\to P$ the unique $\tau$-preserving {\it conditional expectation} from $M$ onto $P$, by $e_P:L^2(M)\to L^2(P)$ the orthogonal projection onto $L^2(P)$ and by $\langle M,e_P\rangle$ the Jones' basic construction of $P\subset M$. We also denote by $P'\cap M=\{x\in M\,:\, xy=yx, \text{ for all } y\in P\}$ the {\it relative commutant} of $P$ in $M$ and by $\mathscr N_{M}(P)=\{u\in\mathscr U(M)\,:\, uPu^*=P\}$ the {\it normalizer} of $P$ in $M$.  For two von Neumann subalgebras $P_1,P_2\subset M$, we denote by $P_1\vee P_2=W^*(P_1\cup P_2)$ the von Neumann algebra generated by $P_1$ and $P_2$.

Let $\omega$ be a free ultrafilter on $\mathbb N$. Consider the C$^*$-algebra $\ell^\infty(\mathbb N, M)=\{ (x_n)_n\subset M\,:\,  \sup_n\| x_n\|<\infty \}$ and its closed ideal $\mathcal I=\{(x_n)_n\in \ell^\infty(\mathbb N,M)\,:\, \lim_{n\to\omega}\|x_n\|_2=0  \}$. The {\it ultrapower} of $M$ is the tracial von Neumann algebra $M^\omega:=\ell^\infty (\mathbb N, M)/\mathcal I$ whose canonical trace is given by $\tau_\omega(x)=\lim_{n\to\omega} \tau (x_n)$, for any $x=(x_n)_{n}\in M^\omega.$

Finally, for any positive integer $n$, we denote by $\overline{1,n}$ the set $\{1,\dots,n\}$.

\subsection{Finite index inclusions of von Neumann algebras}
The Jones index for an inclusion $P\subseteq M$ of II$_1$ factors is the dimension of $L^2(M)$ as a left $P$-module \cite{Jo81}.
Pimsner and Popa defined a probabilistic notion of index for an inclusion $P \subseteq M$ of arbitrary von Neumann
algebras with conditional expectation, which extends 
Jones’ index for inclusions of II$_1$ factors  \cite[Theorem 2.2]{PP86}. Specifically, the inclusion $P \subseteq M $ of tracial von Neumann algebras
is said to have {\it probabilistic index} $[M:P]=\lambda^{-1}$, where
$$
\lambda={\rm inf}\{ \| E_P(x) \|_2^2 \|x\|_2^{-2}\,:\, x\in M_+, x\neq 0 \}.
$$
Here we use the convention that $\frac{1}{0}=\infty.$ 

We continue by recording several basic facts from the literature concerning finite index inclusions of von Neumann algebras which we will use  in the proofs of our main results in Sections \ref{Section.primeness} and \ref{Section.ec.factors.indecomposability}. First, recall that a von Neumann algebra $M$ is called completely atomic if $1$ is an orthogonal sum of minimal projections in $M$. The first three items are essentially contained in \cite{Po95} and we refer the reader to \cite[Proposition 2.3]{CdSS17} for a proof. The fourth item is precisely \cite[Proposition 2.1.15]{Jo81}, while the last one follows directly from the definition.

\begin{proposition}\label{findex}\label{Theorem.findex} Let $N\subset M$ be an inclusion of tracial von Neumann algebras with $[M:N]<\infty$ and $P$ a tracial factor. Then the following hold:

\begin{enumerate}

\item If $p\in N$ is a non-zero projection, then $[pMp:pNp]<\infty$. 
\item If $N$ is a factor and $r\in N'\cap M$ is a non-zero projection, then $[rMr: Nr]<\infty$.

\item If  $\mathcal Z(M)$ is completely atomic, then $\mathcal Z(N)$ is completely atomic.

\item If $M$ and $N$ are factors, then $[M\bar\otimes P: N\bar\otimes P]<\infty$.

\item Let $R\subset N$ be a von Neumann subalgebra satisfying $[N:R]<\infty$. Then $[M:R]=[M:N][N:R]$.

\end{enumerate}
\end{proposition}
We end this section by recalling two more results on finite index inclusions, that will be essential to deriving our main results in Section 7.
\begin{lemma}[\text{\cite[Lemma 3.1]{Po02}}]\label{lemma.finite index.relative.commutant}
Let $P\subset Q\subset M$ be an inclusion of tracial von Neumann algebras such that $P\subseteq Q$ is a finite index inclusion of II$_1$ factors.
Then $[P'\cap M: Q'\cap M]<\infty$.
\end{lemma}

\begin{lemma}[\text{\cite[Lemma 2.3]{CIK13}}]\label{findexunderexpectation} Let $M,N\subseteq P
$ be von Neumann algebras. Let $M_0$ be the von Neumann algebra generated by $E_N(M)$ inside $N$,   $M_0:=E_N(M)''\subseteq N$. If $[P:M]<\infty$, then $[N:M_0]<\infty$.\end{lemma}

\subsection{Popa's intertwining techniques} Almost two decades ago, S. Popa  introduced  in \cite [Theorem 2.1 and Corollary 2.3]{Po03} a powerful analytic criterion for identifying intertwiners between arbitrary subalgebras of tracial von Neumann algebras, see Theorem \ref{corner} below. This technique,  known as \emph{Popa's intertwining-by-bimodules technique},  has played an essential role in the classification of von Neumann algebras program via Popa's deformation/rigidity theory.  

\begin{theorem}\emph{\cite{Po03}} \label{corner} Let $( M,\tau)$ be a  tracial von Neumann algebra and let $ P,  Q\subseteq  M$ be (not necessarily unital) von Neumann subalgebras. 
	Then the following are equivalent:
	\begin{enumerate}
		\item There exist projections $ p\in    P, q\in    Q$, a $\ast$-homomorphism $\theta:p  P p\rightarrow q Q q$  and a partial isometry $0\neq v\in  M $ such that $v^*v\leq p$, $vv^*\leq q$ and $\theta(x)v=vx$, for all $x\in p  P p$.
		\item For any group $\mathcal G\subset \mathscr U( P)$ such that $\mathcal G''=  P$ there is no net $(u_n)_n\subset \mathcal G$ satisfying $\|E_{  Q}(xu_ny)\|_2\rightarrow 0$, for all $x,y\in   M$.
		\item There exist finitely many $x_i, y_i \in  M$ and $C>0$ such that  $\sum_i\|E_{  Q}(x_i u y_i)\|^2_2\geq C$ for all $u\in \mathscr U( P)$.
		\item There exists a non-zero projection $f\in P'\cap \langle M, e_Q \rangle$ such that ${\rm Tr}(f)<\infty$.
	\end{enumerate}
\end{theorem} 
\vskip 0.02in
\noindent If one of the three equivalent conditions from Theorem \ref{corner} holds, then we say that \emph{ a corner of $ P$ embeds into $ Q$ inside $ M$}, and write $ P\prec_{ M} Q$. If we moreover have that $ P p'\prec_{ M} Q$, for any projection  $0\neq p'\in  P'\cap 1_{ P}  M 1_{ P}$, then we write $ P\prec_{ M}^{\rm s} Q$.

For further use, we recall several useful intertwining results for von Neumann subalgebras.

\begin{lemma}\label{lemma.elemetary.facts.intertwinining}
Let $(M,\tau)$ be a tracial von Neumann algebra and $P\subset pMp$, $Q\subset qMq$, $R\subset rMr$ be von Neumann subalgebras, where $p,q,r\in M$ are projections. Then the following hold.
\begin{enumerate}
    
    \item \emph{\cite[Lemma 2.4(2)]{DHI16}} If $Pz\prec_M Q$ for any non-zero projection $z\in \mathscr N_{pMp}(P)'\cap M$, then $P\prec_M^{\rm s} Q$.
    
    \item \emph{\cite[Lemma 2.4(3)]{DHI16}} Assume $P\prec_M Q$. Then there is a non-zero projection $z\in \mathcal N_{pMp}(P)'\cap M$ such that $Pz\prec_M^{\rm s} Q$.

    \item \emph{\cite[Lemma 3.5]{Va08}} If $P\prec_M Q$, then $Q'\cap qMq\prec_M P'\cap pMp$.
    
    \item \emph{\cite[Lemma 3.7]{Va08}} If $P\prec_M Q$ and $Q\prec_M^{\rm s} R$, then $P\prec_M R$.
    \item If $S\subseteq P$ is a finite index von Neumann subalgebra then $S\prec_M R$ if and only if $P \prec_M R$.
\end{enumerate}

\end{lemma}

\begin{lemma}\emph{\cite[Lemma 2.3]{PP86}}\label{Theorem.PP86}
Let $N\subset M$ be an inclusion of tracial von Neumann algebras satisyfing $[M:N]<\infty$. The following hold:

\begin{enumerate}
    \item $M\prec_M^{\rm s}N$.
    
    \item If $\mathcal Z(N)$ is completely atomic, then $M\prec_M Nq$, for any non-zero projection $q\in N'\cap M$.
\end{enumerate}

\end{lemma}

For proofs of (1) and (2) we refer the reader to \cite[Lemma 2.4]{CIK13} and \cite[Lemma 2.11]{Dr19b}, respectively.

 We end this section with the following elementary lemma. For completeness, we also include a short proof and refer the reader to Definition \ref{definition.cocycle.action} for the notion of cocycle action.
\begin{lemma}\label{findexint} Let $\Gamma \curvearrowright^{\sigma,\alpha} A$ be a cocycle action and let $M= A \rtimes_{\sigma, \alpha}\Gamma $ the corresponding twisted crossed product von Neumann algebra. Then   $M \prec_M A$ if and only if $\Gamma$ is finite.  
    
\end{lemma}
\begin{proof}
    If $\Gamma$ is finite then $A\subseteq M$ has finite index and thus Lemma \ref{Theorem.PP86}(1) implies that $M\prec_M A$. 
    
    To see the converse assume $M\prec_M A$. Then part (3) in Theorem \ref{corner} implies the existence of finitely many $x_1, \ldots ,x_n,  y_1, \ldots ,y_n\in M$ and $C>0$ such that $\sum^k_{i=1} \|E_A(x_i u_g y_i)\|^2_2\geq C,\text{ for all }g\in \Gamma$. Approximating $x_i$ and $y_i$, via the Kaplansky Density Theorem, by elements in the $\ast$-algebra generated by the linear span of $A\Gamma$, this inequality further implies the existence of an integer $l\geq k$, a constant $C'>0$, and  $x'_i=a_i u_{g_i}$, $y'_i= b_i u_{h_i}$, where $a_i,b_i \in A$ and $g_i,h_i \in\Gamma$  for all $i\in \overline{1,l}$, such that \begin{equation}\label{ineqint} \sum^l_{i=1} \|E_A(x'_i u_g  y'_i)\|^2_2\geq C',\text{ for all }g\in \Gamma.
    \end{equation}

Now notice that $E_A(x'_i u_g y'_i)= a_i E_A(u_{g_i} u_g u_{h_i}) \sigma_{h_i^{-1}}(b_i)= a_i E_A( \alpha(g_i,g)\alpha(g_ig,h_i)u_{g_igh_i}) \sigma_{h_i^{-1}}(b_i) = \tau(u_{g_ig h_i}) a_i \alpha(g_i,g)\alpha(g_ig,h_i)\sigma_{h_i^{-1}}(b_i)$. Since $\alpha(g_i,g)\alpha(g_ig,h_i)\in \mathscr U(A)$, basic estimates further imply that  $\|E_A(x'_ix y'_i)\|_2\leq \|a_i\|_\infty \|b_i\|_\infty |\tau(u_{g_igh_i}) |= \delta_{g_igh_i,1}  \|a_i\|_\infty \|b_i\|_\infty$. Combining this with relation \eqref{ineqint} we get  \begin{equation*}\sum^l _{i=1} \delta_{g_igh_i,1} \geq \frac{C'}{ \max^l_{i=1} \|a_i\|_\infty \|b_i\|_\infty }>0.\end{equation*} 
This implies $\Gamma \subseteq \{g_i^{-1}h_i^{-1} \,:\, i\in \overline{1,l}\}$, entailing $\Gamma$ is finite.  \end{proof}

\subsection{Relative amenability and weak containment of bimodules}\label{section.amenable}

A  tracial von Neumann algebra $(M,\tau)$ is {\it amenable} if
there exists a net $(\xi_n)_n\in L^2(M)\otimes L^2(M)$ such that $\langle x\xi_n,\xi_n \rangle\to \tau(x)$ and $\|x\xi_n-\xi_n x \|_2\to 0$, for all $x\in M$.
 Connes' celebrated classification of amenable factors \cite{Co76} shows that $M$ is amenable if and only if $M$ is approximately finite dimensional.
 
Next, we recall the notion of relative amenability introduced by 
Ozawa and Popa in \cite{OP07}. Let $(M,\tau)$ be a  tracial von Neumann algebra. Let $p\in M$ be a projection and $P\subset pMp,Q\subset M$ be von Neumann subalgebras. Following \cite[Section 2.2]{OP07}, we say that $P$ is {\it amenable relative to $Q$ inside $M$} if
there exists a net $(\xi_n)_n\in L^2(\langle M,e_Q \rangle)$ such that $\langle x\xi_n,\xi_n \rangle\to \tau(x)$ for all $x\in pMp$ and $\|y\xi_n-\xi_n y \|_2\to 0$, for all $y\in P$.
It is a fact that $P$ is amenable relative to $\mathbb C$ inside $M$ if and only if $P$ is amenable.

Let $M,N$ be tracial von Neumann algebras.  An $M$-$N$-bimodule $_M\mathcal H_N$ is a Hilbert space $\mathcal H$ equipped with two commuting normal unital $*$-homomorphisms $M\to \mathbb B(\mathcal H)$ and $N^{\rm op}\to \mathbb B(\mathcal H)$. If $M=N$, we say, for simplicity, that $_M\mathcal H_M$ is an $M$-bimodule.  For two bimodules $_M\mathcal H_N$ and $_M\mathcal K_N$, we say that $_M\mathcal H_N$ is weakly contained in $_M\mathcal K_N$ if for any $\epsilon>0$, finite sets $F\subset M, G\subset N$ and $\xi\in\mathcal H$, there exist $\eta_1,\dots,\eta_n\in \mathcal K$ such that 
$$
|\langle x\xi y, \xi \rangle -\sum_{i=1}^n  \langle x\eta_i y, \eta_i \rangle| \leq \epsilon,\; \text{ for all } x\in F \text{ and } y\in G. 
$$
Examples of bimodules include the trivial $M$-bimodule $_ML^2(M)_M$ and the coarse $M$-$N$-bimodule $_ML^2(M)\otimes L^2(N)_N$.

Finally, 
it is a fact that a subalgebra $P\subset pMp$ is amenable relative to $Q$ inside $M$ if and only if $_P L^2(pM)_M$ is weakly contained in $_P L^2(\langle M,e_{Q} \rangle)_M$ (see for instance \cite[Theorem 13.4.4]{AP22}).

\section{Popa's Deformation/Rigidity Theory}\label{section.popa.deformation.rigidity}

\subsection{Deformations} 
In the last two decades, Popa's deformation/rigidity theory has led to unprecedented progress in the theory of von Neumann algebras, see the survey papers \cite{Po07,Va10,Io12,Io17}. 
Popa's theory is centered on the remarkable idea of using deformations of II$_1$ factors to locate rigid subalgebras. In Theorem \ref{solid} we will illustrate this principle by presenting Popa's proof \cite{Po06b} of Ozawa's  solidity result \cite{Oz03} for the free group factors. 
First, we make precise the notion of a deformation.
 
\begin{definition} A {\it deformation} of the identity of a tracial von Neumann algebra $(M,\tau)$ is a sequence  of unital, trace preserving, completely positive maps $\phi_n:M\to M$ satisfying $\|\phi_n(x)-x \|_2\to 0$, for all $x\in M.$

A linear map $\phi:M\to M$ is called {\it completely positive} if the amplification $\phi^{(m)}: \mathbb M_m(\mathbb C)\to \mathbb M_m(\mathbb C)$ given by $\phi^{(m)}([x_{i,j}])=[\phi(x_{i,j})]$ is positive, for any $m\ge 1$.
\end{definition}

Before presenting some examples of deformations, we recall some terminology. Let $(M,\tau)$ be a tracial separable von Neumann algebra. A map $\phi: M\to M$ is called {\it subunital} if $\phi(1)\leq 1$ and {\it subtracial} if $\tau\circ\phi\leq \tau$. For a von Neumann subalgebra $P\subset M$, the map $\phi$ is called {\it $P$-bimodular} if $\phi(axb)=a\phi(x)b$, for all $a,b\in P$ and $x\in M$.
Let $P\subset M$ be a von Neumann subalgebra. Following \cite[Section 2]{Po01b}, we say that $M$ has {\it the Haagerup property  relative} to $P$ if there exists a sequence of normal subunital subtracial completely positive $P$-bimodular maps $\phi_n:M\to M$ such that $\| \phi_n(x)-x \|_2\to 0$, for any $x\in M$ and such that  every $\phi_n$ satisfies the following relative compactness property: if $(w_k)_k$ is a sequence of unitaries in $M$ satisfying $\|E_P(aw_ny)\|_2\to 0$, for all $a,b\in M$, then $\| \phi_n(w_k) \|_2\to 0$ when $k\to\infty$.

Next, we continue by presenting two examples of deformations. For a comprehensive list of examples, we
refer the reader to \cite[Section 3]{Io12} (see also \cite[Section 3]{Io17}).
The first one was used by Popa by using the above Haagerup notion to show that certain von Neumann algebras have trivial fundamental group \cite{Po01b}, thereby solving a longstanding question in operator algebras.

\begin{example}
Let $\Gamma\curvearrowright (X,\mu)$ be a probability measure preserving action of a countable group and let $\varphi_n: \Gamma\to\mathbb C$ be a sequence of positive definite functions such that $\varphi_n(g)\to 1$, for all $g\in\Gamma$. Then the sequence $\phi_n:L^\infty(X)\rtimes \Gamma\to L^\infty(X)\rtimes \Gamma$ given by $\phi_n(au_g)=\varphi_n(g)au_g$ is a deformation of the identity of $L^\infty(X)\rtimes\Gamma$.

If $\Gamma$ has the Haagerup property \cite{Ha79}, one can choose $\varphi_n:\Gamma\to\mathbb C$ to satisfy $\varphi_n\in c_0(\Gamma)$, for any $n$. In addition, if $\Gamma=SL_2(\mathbb Z)$, then the the sequence $\phi_n$ is a deformation of the II$_1$ factor $M=L^\infty(\mathbb T^2)\rtimes SL_2(\mathbb Z)$ which witnesses the fact that $M$ has the Haagerup property relative to $L^\infty(\mathbb T^2)$. 
This fact is an essential ingredient in Popa's proof that $M$ has trivial fundamental group.
\end{example}


Now, we now turn our attention to the foundational notions of {\it malleable} and {\it s-malleable deformations} introduced by Popa in \cite{Po01b, Po03} which consists of a special path of completely positive maps. This
novel concept was used with great effect in classifying large classes of factors.

\begin{definition}\label{def:malleable}
Let $(M,\tau)$ be a tracial von Neumannn algebra. A pair $(\tilde M, (\alpha_t)_{t\in\mathbb R})$ is called a {\it malleable deformation} of $M$ if the following conditions hold:
\begin{enumerate}
    \item $(\tilde M,\tilde\tau)$ is a tracial von Neumann algebra such that $M\subset\tilde M$ and $\tau=\tilde\tau_{|M},$ 
    
    \item $(\alpha_t)_{t\in\mathbb R}\subset {\text Aut}(\tilde M,\tilde\tau)$ is a $1$-parameter group with $\lim_{t\to 0}\|\alpha_t(x)-x\|_2=0$, for any $x\in\tilde M,$ 
    
    \item $\alpha_t$ does not converge uniformly to the identity on $(M)_1$ as $t\to 0$.

\end{enumerate}

In addition, if also the following condition holds

\begin{enumerate}
    \item [(4)] There exists $\beta\in {\text Aut}(\tilde M,\tilde\tau)$ that satisfies $\beta_{|M}=\text{Id}_M$, $\beta^2=\text{Id}_{\tilde M}$ and $\beta\alpha_t=\alpha_{-t}\beta$, for any $t\in\mathbb R$.

\end{enumerate}

then $(\tilde M, (\alpha_t)_{t\in\mathbb R},\beta)$ is called an {\it s-malleable deformation} of $M$.

\end{definition}


There are several natural classes of von Neumann algebras that admit (s-)malleable deformations:
\begin{itemize}
    \item Free group factors, and more generally, amalgamated free product von Neumann algebras \cite{Po86, Po06b, IPP05}.

    \item HNN extensions von Neumann algebras \cite{FV10}.

    \item Wreath products von Neumann algebras \cite{Po03, Po05, Io06}.

    \item Von Neumann algebras that are constructed from trace preserving action of groups that admit unbounded cocycles valued into various orthogonal representations,
 \cite{Si10}.

    \item Von Neumann algebras that admit certain unbounded closable derivations \cite{DI12}.
\end{itemize}


Our second deformation that we concretely exemplify is the malleable deformation of free groups factors which was used by Popa \cite{Po06b} to prove Ozawa's solidity result \cite{Oz03} for free group factors.

\begin{example}\label{example.deformation.free.groups}
Free group factors admit natural malleable deformations  as follows \cite{Po06b}. For simplicity, we recall the deformation of the group factor $M=L(\mathbb F_n)$ with $2\leq n\leq\infty$.
Let $\{a_k\}_{k=1}^n\cup \{b_k\}_{k=1}^n$ be the generators of $\mathbb F_{2n}$ and denote  
$\tilde M=M*M$.
By viewing $\mathbb F_n$ as the subgroup of $\mathbb F_{2n}$ generated by $\{a_k\}_{k=1}^n$, we obtain an embedding of $M$
inside  $\tilde M$. We denote still by $\{a_k\}_{k=1}^n\cup \{b_k\}_{k=1}^n$ the canonical unitaries in $\tilde M$. We now construct a 1-parameter group of automorphisms $(\alpha_t)_{t\in\mathbb R}$ of $\tilde M$ as follows. Let $\{h_k\}_{k=1}^n\subset\tilde M$ be hermitian elements such that $b_k=\exp (ih_k)$, for any $k\in\overline{1,n}$.

For every $t\in\mathbb R$, we define a trace preserving automorphism $\alpha_t$ of $\tilde M$ by letting:
$$
\alpha_t(a_k)=\exp(ith_k)a_k, \;\;\; \alpha_t(b_k)=b_k, \;\;\;\text{for every }k\in\overline{1,n}.
$$

In fact, $M$ admits an s-malleable deformation by defining a trace preserving automorphism $\beta$ of $\tilde M$ by letting:
$$
\beta(a_k)=a_k, \;\;\; \beta(b_k)=b_k^*, \;\;\;\text{for every }k\in\overline{1,n}.
$$
\end{example}


\subsection{Deformation vs. rigidity}

An important concept in deformation/rigidity theory and a main source of rigidity is the notion of property (T) for von Neumann algebras introduced by Connes and Jonnes \cite{CJ85} and its relative version  for inclusions of von Neumann algebras defined by Popa \cite{Po01b}.

\begin{definition}
An inclusion of tracial von Neumann algebras $P\subset M$ has the {\it relative property (T)} if any deformation $\phi_n:M\to M$ of the identity must converge uniformly to the identity on the unit ball of  $P$.
A tracial von Neumann algebra $M$ has {\it property (T)} if the inclusion $M\subset M$ has relative property (T).
\end{definition}

Next, we exemplify why relative property (T) should be thought of as rigidity by noticing that any property (T) von Neumann algebra $M$ cannot have ``lots" of deformations. For instance, if $M$ has the Haagerup property relative to $\mathbb C$, then $M$ is not diffuse. Indeed, using the Haagerup property we let $\phi_n:M\to M$ be a sequence of normal subunital subtracial completely positive maps such that $\|\phi_n(x)-x\|_2\to 0$, for any $x\in M$, and every $\phi_n$ satisfies $\| \phi_n(w_k) \|_2\to 0$ for any  sequence of unitaries $(w_k)_k$ in $M$  that weakly goes to $0$. Since property (T) can also be characterized by using completely positive maps that are only subunital and subtracial (see \cite{Po01b}), we obtain that $\phi_n$ must converge uniformly to the identity on the unit ball of $M$. Hence, there is $n_0\ge 1$ such that $\|\phi_{n_0}(x)-x \|_2\leq 1/2$ for any $x\in\mathcal U(M)$. If we assume by contradiction that $M$ is not diffuse, then there is a sequence of unitaries $(w_k)_k$ in $M$  that weakly goes to $0$, and hence, $\| \phi_{n_0}(w_k) \|_2\to 0$. This gives the contradiction that $1=\lim_{k\to\infty} \| \phi_{n_0}(w_k)-w_k \|_2\leq 1/2.$

The natural correspondence between completely positive maps and bimodules provides a  reformulation of relative property (T) in terms of bimodules. 
First, recall that if $\mathcal H$ is an $M$-bimodule, then a vector $\xi\in  M$ is  $P${\it-central}
 if $x\xi=\xi x$, for all $x\in P$, and   {\it tracial} if $\langle x\xi,\xi \rangle= \langle \xi x, \xi \rangle=\tau(x)$, where $M$ is endowed with the trace $\tau$. A net of vectors $(\xi_n)_n\subset \mathcal H$ is called {\it $P$-almost central} if $\| x\xi_n-\xi_n x \|_2\to 0$, for any $x\in P$. Following \cite{Po01b}, the inclusion $P\subset M$ has relative property (T) if and only if any $M$-bimodule without central $P$-vectors does not admit a net $(\xi_n)_n$ of tracial, $M$-almost central vectors. Hence, relative property (T) requires that all $M$-bimodules satisfy a certain spectral gap condition.

A remarkable discovery of Popa \cite{Po06a,Po06b} is that one can obtain rigidity by only using that certain bimodules have spectral gap. For instance, if $P\subset M$ is a non-amenable subalgebra, then there does not exist a net $(\xi_n)_n\in L^2(M)\otimes L^2(M)$ of tracial, $P$-almost central vectors; this follows from the fact that $P$ is non-amenable relative to $\mathbb C$ inside $M$, see Section \ref{section.amenable}. 
To explain the spectral gap terminology, we note that Connes' theorem \cite{Co76} gives that a II$_1$ factor $P$ is non-amenable if and only if there exist a finite set $S\subset P$ and $k>0$ such that 
$$
\|\xi\|_2\leq  k \sum_{y\in S} \| y\xi-\xi y\|_2, \;\; \text{ for all }\xi\in L^2(P)\otimes L^2(P).
$$
This principle of using spectral gap from non-amenable subalgebras is illustrated below in the proof of Theorem \ref{solid}.


\subsection{Solidity of Free Group Factors} 

In this section we present Popa's  deformation/rigidity theory proof \cite{Po06b} of Ozawa's solidity result for the free group factors.

\begin{theorem} \emph{\cite{Oz03}}\label{solid}
Let $M=L(\mathbb F_n)$, for some $2\leq n\leq\infty$. Then $M$ is {\it solid}, i.e., for every diffuse von Neumann subalgebra $A\subset M$, the relative commutant $A'\cap M$ is amenable. 
\end{theorem}

The proof uses the {free malleable deformation} $(\tilde M, (\alpha_t)_{t\in\mathbb R})$ from Example \ref{example.deformation.free.groups}. Recall that $\tilde M=M*M$. 
We will use the fact that the contraction $E_M\circ\alpha_t :L^2(M)\to L^2(M)$ is a compact operator, for any $t>0$.
In fact, the following weaker property will suffice for our purposes:

\begin{lemma}\label{compact}
If $(u_n)_n\subset\mathcal U(M)$ is a sequence that  converges weakly to $0$, then for any $t>0$ we have $\lim_{n\rightarrow\infty}\|E_M(\alpha_t(u_n))\|_2=0$.
\end{lemma}

{\it Proof.}
Denote $\rho(t)=\frac{\sin(\pi t)}{\pi t}$ and notice that $0<\rho(t)<1$, for all $t>0$. Note that $\exp(i h_k)$ can be seen as the canonical generating unitary of $L(\mathbb Z)$ for any $k\in\overline{1,n}$. Thus, by identifying $L(\mathbb Z)$ with $L^\infty(\mathbb T)$ we derive that 
\begin{equation*}\label{trace}
\tau(\exp(ith_k))=\frac{1}{2\pi}\int_{-\pi}^{\pi}\exp(it\theta)\;\text{d}\theta=\rho(t), \;\;\;\; \text{ for all }t>0 \text{ and }k\in \overline{1,n}.
\end{equation*}
This implies that
\begin{equation}\label{sss1}
E_M(\alpha_t(u_g))=\rho(t)^{|g|}u_g,\;\;\; \text{for all } t\in\mathbb R \text{ and } g\in\mathbb F_n.    
\end{equation}
Here,  $|g|$ denotes the {word length} of $g\in\mathbb F_n$ with respect to the generating set $\{a_k\}_{k=1}^n$.

Consider the Fourier expansion $u_n=\sum_{g\in\mathbb F_n}\tau(u_nu_g^*)u_g$. Then by formula \ref{sss1} we get that $E_M(\alpha_t(u_n))=\sum_{g\in\mathbb F_n}\rho(t)^{|g|}\tau(u_nu_g^*)u_g$ and thus $\|E_M(\alpha_t(u_n))\|_2^2=\sum_{g\in\mathbb F_n}\rho(t)^{2|g|}|\tau(u_nu_g^*)|^2.$ Thus, if $N\geq 1$ is an integer, then using that $\sum_{g\in\mathbb F_n}|\tau(u_nu_g^*)|^2=\|u_n\|_2^2=1$, we get that $$\|E_M(\alpha_t(u_n))\|_2^2\leq \sum_{|g|\leq N}|\tau(u_nu_g^*)|^2+\rho(t)^{2N}\sum_{|g|\geq N}|\tau(u_nu_g^*)|^2\leq \sum_{|g|\leq N}|\tau(u_nu_g^*)|^2+\rho(t)^{2N}.$$
Since the set $\{g\in F_n \,:\,|g|\leq N\}$ is finite and $\lim_{n\rightarrow\infty}\tau(u_nu_g^*)=0$, for all $g\in\mathbb F_n$, we conclude that $\limsup_{n\rightarrow\infty}\|E_M(\alpha_t(u_n))\|_2^2\leq\rho(t)^{2N}$, for all $N\geq 1$. Since $0<\rho(t)<1$, the conclusion of the corollary follows.
\hfill$\blacksquare$

The source of rigidity in the proof of Theorem \ref{solid} is given by the following spectral gap result.  

\begin{lemma}\label{bimodule}
The Hilbert $M$-bimodule $L^2(\tilde M)\ominus L^2(M)$ is isomorphic to an infinite multiple of the coarse $M$-bimodule, $(L^2(M)\bar{\otimes}L^2(M))^{\oplus_{\infty}}$. 

In particular, if $B\subset M$ is a non-amenable von Neumann subalgebra, then the trivial $B$-bimodule is not weakly contained in the $B$-bimodule $L^2(\tilde M)\ominus L^2(M)$.
\end{lemma}

{\it Proof.} 
 Let $S\subset \mathbb F_{2n}$ be the set of elements $g\in\mathbb F_n$ whose reduced form begins and ends with a non-zero power of $b_k$ for some $k\in\overline{1,n}$. Since $L^2(\tilde M)\ominus L^2(M)=\bigoplus_{g\in S}\overline{{\rm sp} (Mu_gM)}$,  in order to prove the assertion, it suffices to show that $\overline{{\rm sp} (Mu_gM)}\cong L^2(M)\bar{\otimes}L^2(M)$, as Hilbert $M$-bimodules, for any $g\in S$.

If $g\in S$, then $g^{-1}\mathbb F_ng\cap\mathbb F_n=\{e\}$. Hence, $\tau(u_g^*u_hu_gu_k)=\delta_{g^{-1}hg,k^{-1}}=\delta_{h,e}\delta_{k,e}=\tau(u_h)\tau(u_k)$, for all $h,k\in\mathbb F_n$. This implies $\tau(u_g^*au_gb)=\tau(a)\tau(b)$, for all $a,b\in M$, and further that $$\text{$\langle xu_gy,zu_gt\rangle=\tau(u_g^*z^*xu_gyt^*)=\tau(z^*x)\tau(yt^*)=\langle x\otimes y,z\otimes t\rangle_{L^2(M)\bar{\otimes}L^2(M)} $, for all $x,y,z,t\in M$.}$$ Thus, $x\otimes y\mapsto xu_gy$ extends to an isomorphism of Hilbert $M$-bimodules $L^2(M)\bar{\otimes}L^2(M)\cong \overline{{\rm sp} (Mu_gM)}$.

In order to prove the last part of the lemma, let $B\subset M$ be a non-amenable von Neumann subalgebra. First, note that the $B$-bimodule $L^2(M)\otimes L^2(M)$ is weakly contained in the coarse $B$-bimodule $L^2(B)\otimes L^2(B)$. This follows from the fact that any left (respectively, right) $B$-module is contained in $L^2(B)^{\oplus\infty}$ as a left (respectively, right) $B$-module (see, for instance, \cite[Proposition 8.2.3]{AP22}). Next, we derive that the $B$-bimodule $(L^2(M)\bar{\otimes}L^2(M))^{\oplus{\infty}}$ is weakly contained in the $B$-bimodule $(L^2(B)\bar{\otimes}L^2(B))^{\oplus{\infty}}$, and hence, in the coarse $B$-bimodule $L^2(B)\bar{\otimes}L^2(B)$. Finally, if we assume by contradiction that the trivial $B$-bimodule is  weakly contained in the $B$-bimodule $L^2(\tilde M)\ominus L^2(M)$, we derive from the first part of the lemma that  the trivial $B$-bimodule is  weakly contained in the coarse $B$-bimodule $L^2(B)\bar{\otimes}L^2(B)$. This implies that $B$ is amenable, contradiction.
This proves the lemma.
\hfill$\blacksquare$

{\it Proof of Theorem \ref{solid}.} Let $A\subset M$ be a diffuse von Neumann subalgebra and let $(u_k)_k\subset\mathcal U(A)$ be a sequence which converges weakly to $0$. Denote $B=A'\cap M$.

Put $t_n=1/2^n$, for every $n\geq 1$. By Lemma \ref{compact} we can find a subsequence $(v_n)_n$ of $(u_k)_k$ such that $\lim_{n\rightarrow\infty}\|E_M(\alpha_{t_n}(v_n))\|_2=0$. By letting  $\xi_n=\alpha_{t_n}(v_n)-E_M(\alpha_{t_n}(v_n))$, we  claim that \begin{equation}\label{ss1}
\text{$\lim_{n\rightarrow\infty}\|[x,\xi_n]\|_2=0$, for every $x\in B$,}\end{equation}
\begin{equation}\label{ss2}\text{$\lim_{n\rightarrow\infty}\langle x\xi_n,\xi_n\rangle=\tau(x)$, for every $x\in \tilde M$.}\end{equation}

To prove \eqref{ss1}, note that for any $x\in B$ we have $\|[x,E_M(\alpha_{t_n}(v_n)) ]\|_2\leq 2 \|x\| \|E_M(\alpha_{t_n}(v_n))\|_2$ and
$$
\|[x,\alpha_{t_n}(v_n)]\|_2=\|[\alpha_{-t_n}(x),v_n]\|_2=\|[\alpha_{-t_n}(x)-x,v_n]\|_2 \leq \|\alpha_{-t_n}(x)-x\|_2.
$$

To prove \eqref{ss2}, note that
 $\lim_{n\rightarrow\infty}\|E_M(\alpha_{t_n}(v_n))\|_2=0$ implies that
 for every $x\in \tilde M$, we have
 $$
 \lim_{n\rightarrow\infty}\langle x\xi_n,\xi_n\rangle=\lim_{n\rightarrow\infty}\langle x\alpha_{t_n}(v_n),\alpha_{t_n}(v_n)\rangle=\tau(x).
 $$

Combining \eqref{ss1} and \eqref{ss2} it follows that 
the trivial $B$-bimodule is weakly contained in the $B$-bimodule $L^2(\tilde M)\ominus L^2(M)$.
Hence, by Lemma \ref{bimodule} it follows that $B$ is amenable.
\hfill$\blacksquare$

\section{Wreath-like Products and Wreath-like von Neumann Algebras}\label{section.wreath}

For further use, we recall in this section  the main result of our very recent work \cite{CDI22}. This asserts that property (T) II$_1$ factors form an \emph{embedding universal} family, i.e., every separable tracial von Neumann algebra embeds into a property (T) II$_1$ factor, see \cite[Theorem A]{CDI22}.  This result is proved using the so-called wreath-like product von Neumann algebras. In turn, these are built from wreath-like product groups which were introduced and studied in \cite{CIOS21} through the lenses of geometric group theory.  To provide some context we first recall these notions along with some of their basic properties. 

\subsection{Wreath-like product groups} Following \cite{CIOS21} let $A$ and $B$ be countable groups. Then $G$ is a \emph{regular wreath-like product} of $A$ and $B$ if it can be realized as a group  extension 

\begin{equation}\label{regwreathlike1'''}
    1 \rightarrow \bigoplus_{b\in B} A_b\hookrightarrow G \overset{\varepsilon}{\twoheadrightarrow} B \rightarrow 1 
 \end{equation}
 which satisfies the following properties:
 \begin{enumerate}
     \item [a)] $A_b\cong A$ for all $b\in B$, and 
     \item [b)] the conjugation action of  $G$ on $\bigoplus_{b\in B} A_b$  permutes the direct summands according to the rule \begin{equation*}g A_b g^{-1}= A_{\varepsilon(g)b}\text{, for all }g\in G, b\in B.\end{equation*}
 \end{enumerate}
 The class of all such wreath-like groups is denoted by $\mathcal W\mathcal R(A, B)$. When the extension \eqref{regwreathlike1'''} splits, $G$ is the classical wreath product of $A$ and $B$,  $G = A\wr B$. 
 
 \vskip 0.03in
 Next we recall the concept of a cocycle semidirect product group; see \cite[pages 104-105]{Br82}.
\begin{definition}
A {\it cocycle action} of a group $B$ on a group $A$ is a pair $(\alpha, v)$ consisting of two maps $\alpha:B\rightarrow\text{Aut}(A)$ and $v:B\times B\rightarrow A$ which satisfy the following 
\begin{enumerate}
\item $\alpha_b\alpha_c=\text{Ad}(v_{b,c})\alpha_{bc}$, for every $b,c\in B$, 
\item $v_{b,c}v_{bc,d}=\alpha_b(v_{c,d})v_{b,cd}$, for every $b,c,d\in B$, and
\item $v_{b,1}=v_{1,b}=1$, for every $b\in B$.
\end{enumerate}
\end{definition}

\begin{definition}\label{csproduct}
 Let $(\alpha, v)$ be a cocycle action of a group $B$ on group $A$.   Then the set $A\times B$ endowed with the unit $1=(1,1)$ and the multiplication operation $(x,b)\cdot (y,c)=(x\alpha_b(y)v_{b,c},bc)$ is a group, denoted $A\rtimes_{\alpha,v}B$, and called the {\it cocycle semidirect product} group. 
Moreover, we have a short exact sequence $1\rightarrow A\xrightarrow{i} A\rtimes_{\alpha,v}B\xrightarrow{\gamma} B\rightarrow 1$, where $i(a)=(a,1)$ and $\gamma(a,b)=b$.\end{definition}

 \begin{remark}\label{remark.semidirect.product}
 It is not hard to prove that a group $G$ belongs to $\mathcal W\mathcal R(A,B)$ if and only if it is isomorphic to $A^{(B)}\rtimes_{\alpha,v}B$, for a cocycle action $(\alpha,v)$ on $B$ on $A^{(B)}$ such that $\alpha_b(A_c)=A_{bc}$, for every $b,c\in B$; see \cite[Corollary 2.11]{CDI22}.
 \end{remark}

Wreath-like product groups admit a special cocycle semidirect product decomposition. Let $\sigma$ be the {\it shift} action of $B$ on $A^B=\prod_{b\in B}A$ given by $\sigma_b(x)=(x_{b^{-1}c})_{c\in B}$, for every $x=(x_c)_{c\in B}\in A^B$ and $b\in B$. Note that $\sigma$ leaves invariant the normal subgroup $A^{(B)}=\bigoplus_{b\in B}A$ of $A^B$. The following result has been established in  \cite{CIOS21}; see \cite[Lemmas 2.12 and 2.13]{CDI22} for proofs.

\begin{proposition}\emph{\cite{CIOS21,CDI22}}\label{special'} Let $A,B$ be groups. Then  $G\in\mathcal W\mathcal R(A,B)$ if and only if there is a function $\rho:B\rightarrow A^B$ such that setting $v_{b,c}:=\rho_b\sigma_b(\rho_c)\rho_{bc}^{-1}\in A^{(B)}$, for every $b,c\in B$, and $\rho_1=1$, and letting  $\alpha_b:=\emph{Ad}(\rho_b)\sigma_b\in\emph{Aut}(A^{(B)})$, for every $b\in B$, we have that  $G\cong A^{(B)}\rtimes_{\alpha,v}B$.
\end{proposition}



\subsection{Tracial wreath-like product von Neumann algebras}By analogy with wreath-like product groups there is a notion of wreath-like product for tracial von Neumann algebras. First we need to briefly recall the notion of tracial cocycle crossed product von Neumann algebra.   \begin{definition}\label{definition.cocycle.action}
A {\it cocycle action} of a group $B$ on a tracial von Neumann $(M,\tau)$ is a pair $(\beta, w)$ consisting of two maps $\beta:B\rightarrow\text{Aut}(M)$ and $w:B\times B\rightarrow\mathscr U(M)$ which satisfy the following 
\begin{enumerate}
\item $\beta_b\beta_c=\text{Ad}(w_{b,c})\beta_{bc}$, for every $b,c\in B$, 
\item $w_{b,c}w_{bc,d}=\beta_b(w_{c,d})w_{b,cd}$, for every $b,c,d\in B$, and
\item $w_{b,1}=w_{1,b}=1$, for every $b\in B$.
\end{enumerate}
\end{definition}

\begin{definition}
Let $(\beta,w)$ be a cocycle action of a group $B$ on a tracial von Neumann algebra $(M,\tau)$.
The {\it cocycle crossed product} von Neumann algebra $M\rtimes_{\beta,w}B$ is a tracial von Neumann algebra which is  generated by a copy of $M$ and unitary elements $\{u_b\}_{b\in B}$ such that
$u_bxu_b^*=\beta_b(x)$, $u_bu_c=w_{b,c}u_{bc}$ and $\tau(xu_b)=\tau(x)\delta_{b,1}$, for every $b,c\in B$ and $x\in M$.
\end{definition}

\begin{definition}
Let $(M,\tau)$ be a tracial von Neumann algebra and $B$ be a group. A tracial von Neumann algebra $N$ is said to be a {\it wreath-like product of $M$ and $B$} if it is isomorphic to $M^B\rtimes_{\beta,w}B$, where $(\beta,w)$ is a cocycle action of $B$ on $M^B$ such that $\beta_b(M^c)=M^{bc}$, for every $b,c\in B$. We denote by  $\mathcal W\mathcal R(M,B)$ the class of all wreath-like products of $M$ and $B$.
\end{definition}

\begin{example}
If $G\in\mathcal W\mathcal R(A,B)$, then $L(G)\in\mathcal W\mathcal R(L(A),B)$; see \cite[Example 3.2]{CDI22}.


\end{example}

\begin{notation}\label{notatie}
Let  $(M,\tau)$ be a tracial von Neumann algebra and $B$ be a group. We denote by 
\begin{enumerate}
\item $\gamma:\mathscr U(M)^{(B)}\rightarrow\mathscr U(M^B)$ the homomorphism given by $\gamma((x_b)_{b\in B})=\otimes_{b\in B}x_b$.
\item $\eta:\mathscr U(M)^B\rightarrow\text{Aut}(M^B)$ the homomorphism given by  $\eta((y_b)_{b\in B})=\otimes_{b\in B}\text{Ad}(y_b)$.
\item $B\curvearrowright^{\sigma}\mathscr U(M)^B$ the shift action of $B$ (which preserves the subgroup $\mathscr U(M)^{(B)}<\mathscr U(M)^B$).
\item $B\curvearrowright^{\lambda} M^B$ the Bernoulli shift action given by $\lambda_b(x)=\otimes_{c\in B}x_{b^{-1}c}$, for $x=\otimes_{c\in B}x_c\in M^B$.
\end{enumerate}
\end{notation}

With this notation, we have:


\begin{lemma}\label{wrvN1}
 Let $(M,\tau)$ be a tracial von Neumann algebra and $B$ a group. Let $\xi:B\rightarrow\mathscr U(M)^B$ be a map such that $\xi_b\sigma_b(\xi_c)\xi_{bc}^{-1}\in\mathscr U(M)^{(B)}$, for every $b,c\in B$. Define  $\beta_b=\eta(\xi_b)\lambda_b\in\emph{Aut}(M^B)$ and $w_{b,c}=\gamma(\xi_b\sigma_b(\xi_c)\xi_{bc}^{-1})\in\mathscr U(M^B)$, for every $b,c\in B$. Then $(\beta,w)$ is a cocycle action of $B$ on $M^B$ and $M^B\rtimes_{\beta,w}B\in\mathcal W\mathcal R(M,B)$.
\end{lemma}

Using this lemma in combination with Proposition \ref{special'} it was shown in \cite{CDI22} that given $G\in\mathcal W\mathcal R(A,B)$,  
  any homomorphism $\pi:A\rightarrow\mathscr U(M)$, where $(M,\tau)$ is  a tracial von Neumann algera, extends to a homomorphism $\widetilde\pi:G\rightarrow\mathscr U(N)$,  for some $N\in\mathcal W\mathcal R(M,B)$. 

More precisely,  using Proposition \ref{special'}, we write $G=A^{(B)}\rtimes_{\alpha,v}B$, where $(\alpha,v)$ is a cocycle action of $B$ on $A^{(B)}$ given by $\alpha_b=\text{Ad}(\rho_b)\sigma_b$ and $v_{b,c}=\rho_b\sigma_b(\rho_c)\rho_{bc}^{-1}$, for some map $\rho:B\rightarrow A^B$. Then we have the following:

\begin{proposition}\label{extend'}
Let  $\pi:A\rightarrow\mathscr U(M)$ be a homomorphism, where $(M,\tau)$ is a tracial von Neumann algebra. Define $\xi:=\pi^B(\rho_b)\in\mathscr U(M)^B$, for every $b\in B$. 
Then $\xi_b\sigma_b(\xi_c)\xi_{bc}^{-1}\in\mathscr U(M)^{(B)}$, for every $b,c\in M$.
Define  $\beta_b=\eta(\xi_b)\lambda_b\in\emph{Aut}(M^B)$ and $w_{b,c}=\gamma(\xi_b\sigma_b(\xi_c)\xi_{bc}^{-1})\in\mathscr U(M^B)$, for every $b,c\in B$. 
Then $(\beta,w)$ is a cocycle action of $B$ on $M^B$,  $N:=M^B\rtimes_{\beta,w}B\in\mathcal W\mathcal R(M,B)$ and there is a homomorphism $\widetilde \pi:G\rightarrow\mathscr U(N)$ given by $\widetilde\pi(x)=\gamma(\pi^B(x))=\otimes_{b\in B}\pi(x_b)$ and $\widetilde\pi(e,c)=u_c$, for every $x=(x_b)_{b\in B}\in A^{(B)}$ and $c\in B$.

\end{proposition}

Finally, using Propositions \ref{special'} and \ref{extend'} and the fact that any acylindrically hyperbolic group has plenty of wreath-like quotients (see \cite[Theorems 4.20]{CIOS21}) and that there exist plenty of property (T) wreath-like product groups (see \cite[Theorems 6.9]{CIOS21}), the following result was shown in \cite[Theorem A]{CDI22}. 


\begin{theorem}\label{theorem.cdi}
Let  $(M,\tau)$ be any separable von Neumann algebra. Then the following hold: 

\begin{enumerate}

\item For every acylindrically hyperbolic group $H$, $M$ embeds into a II$_1$ factor $N$ which is generated by a representation $\pi:H\rightarrow\mathscr U(N)$. Thus, if $H$ has property (T), then $N$ has property (T).
\item $M$ embeds into a property (T), wreath-like product II$_1$ factor  $P\in \mathcal{WR}(Q, B)$ where $B$ is a hyperbolic property (T) group, $\emph{Out}(P)=1$, and $\mathcal F(P)=\{1\}$.  More precisely,  $Q$ can be taken to be $Q= (M\ast L(\mathbb F_2))\overline\otimes R$, where $R$ is the hyperfinite II$_1$ factor.  
\end{enumerate}
\end{theorem}

We now explain the terminology used in the theorem. The class of acylindrically hyperbolic groups includes all non-elementary hyperbolic and relatively hyperbolic groups, mapping class groups of closed surfaces of non-zero genus and ${\rm Out}(\mathbb F_n)$, for $n\ge 2.$ We refer the reader to
\cite[Section 3.2]{CIOS21} and to the survey [Osi18] for the precise definition of acylindrically hyperbolic groups and for more details.

For a II$_1$ factor $P$, we denote by $\text{Out}(P)=\text{Aut}(P)/\text{Inn}(P)$ the {\it outer automorphism group of $P$} and by $\mathcal F(P)=\{\tau(e)/\tau(f)\mid\text{$e,f\in P$ projections, $ePe\cong fPf$}\}$
the {\it fundamental group of $P$} \cite{MvN43}.

To this end we make several remarks regarding the previous theorem. Part (1) in Theorem \ref{theorem.cdi} should be viewed as a von Neumann algebraic analogue of the SQ-universality property for (acylindrically) hyperbolic groups established be Olshanskii, \cite{Ol95}, Delaznt \cite{De96}, and Dahmani-Guirardel-Osin \cite{DGO11}.  Recall that a countable group $H$ is called SQ-universal if every countable group embeds into a quotient of $H$. Since property (T) passes to quotients, by taking $H$ to be a hyperbolic group with property (T) it follows that every countable group
embeds into a countable group with property (T). Therefore, Theorem \ref{theorem.cdi} provides an analogue of this fact for II$_1$ factors.

Part (1) applies in particular to icc cocompact lattices $H$ in any  rank one simple real Lie group with finite center (e.g. $Sp(n,1)$, $n\ge 2$), as such $H$ are hyperbolic. Hence, Theorem \ref{theorem.cdi}(1) implies that the family of II$_1$ factors generated by representations of $H$ is embedding universal. This is in sharp contrast with the higher rank case since the work of Peterson \cite{Pe14}  (see also \cite{Be06,BH19})  shows that if $G$ is any icc lattice in higher rank simple real Lie group with finite center (e.g. $\text{SL}_m(\mathbb R)$, $m\geq 3$), then $L(G)$ is the only II$_1$ factor generated by a representation of $G$.

To put Theorem \ref{theorem.cdi}(2) into a better perspective, note that the first examples of II$_1$ factors $P$ with $\mathcal F(P)=\{1\}$ have been obtained by Popa in \cite{Po01b} and the first examples of II$_1$ factors $P$ with $\text{Out}(P)=\{e\}$ and $\mathcal F(P)=\{1\}$ were obtained in \cite{IPP05}. Note that none of these II$_1$ factors have property (T),  although Popa's strengthening of Connes' rigidity conjecture (see \cite[Section 3]{Po07}) predicts that $\text{Out}(L(G))=\{e\}$ and $\mathcal F(L(G))=\{1\}$, whenever $G$ is an icc property (T) group with $\text{Out}(G)=\{e\}$ and no characters. This conjecture has been confirmed for an uncountable class of groups in \cite[Corollary 2.7]{CIOS21} and thus showing that the class $\bm{\mathcal {T}}$ of all II$_1$ factors $P$ with property (T) which satisfy $\text{Out}(P)=\{e\}$ and $\mathcal F(P)=\{1\}$ is uncountable. Theorem \ref{theorem.cdi}(2) shows that $\bm{\mathcal T}$ is in fact embedding universal.

\section{Primeness Results for Wreath-like Product von Neumann Algebras}\label{Section.primeness}

In this section we show that many of the wreath-like product von Neumann algebras introduced in the prior section are (virtually) prime; see Definition \ref{def:vprime} and Theorem \ref{primewlp}. 

We start with two preliminary results on intertwining von Neumann subalgebras in cocycle crossed product von Neumann algebras. These are modest extensions of prior results from \cite{CIK13,CIOS21}, but we include detailed proofs for the readers' convenience.  

\begin{notation}\label{com}Let $G\curvearrowright^{\sigma,\alpha} (Q,\tau)$ be a trace preserving cocycle action on a tracial von Neumann algebra  $(Q,\tau)$. Let $\pi:G\rightarrow H$ be a group epimorphism.  Let $M=Q\rtimes_{\sigma,\alpha} G$.
Following \cite[Section 2]{CIK13}, we consider the $\ast$-embedding $\Delta^\pi: M \hookrightarrow M\overline\otimes L(H)=:\tilde M$ given by $\Delta^\pi(a u_g )= au_g \otimes v_{\pi(g)}$ for every $a\in M$ and $g\in G$. Here we have denoted by $(v_h)_{h\in H}\subset L(H)$ the canonical group unitaries. When $G=H$ and $\pi={\rm id}$, we denote $\Delta^\pi$ simply by $\Delta$.
\end{notation}

Our first result is a straightforward extension of \cite[Proposition 3.4]{CIK13}. 
\begin{proposition}\emph{\cite{CIK13}}\label{CIK13} Assume the Notation \ref{com}  and let $p\in M$ be a projection. Then for any subgroup $K<H$ and any von Neumann subalgebra $P\subseteq pMp$ satisfying $\Delta^\pi (P)\prec_{\tilde M} M\overline\otimes L(K)$, we have that $P \prec_M Q\rtimes_{\sigma, \alpha} \pi^{-1}(K)$. 
\end{proposition}

\begin{proof}  Since $\Delta^\pi (P)\prec_{\tilde M} M\overline\otimes L(K)$, using Theorem \ref{corner} and $\|\cdot\|_2$-approximations one can find group elements $g_1, \ldots g_n, h_1 \ldots, h_n\in H$ and a scalar $C>0$ such that for all $x\in \mathscr U (P)$ we have
\begin{equation}\label{weakmixineq}
\sum^n_{i=1} \|E_{M\overline \otimes L(K)}( (1\otimes v_{g_i})\Delta^\pi(x) (1\otimes v_{h_i}))\|_2\geq C.    
\end{equation}
Now pick elements $k_i,l_i \in G$ satisfying $\pi(k_i )=g_i$ and $\pi(l_i)=h_i$, for all $i\in\overline{1,n}$. Using this, together with the relation $E_{M\overline \otimes L(K)}\circ \Delta^\pi= \Delta^\pi\circ E_{Q \rtimes_{\sigma,\alpha} \pi^{-1}(K)}$, we can see that

\begin{equation*}\begin{split}E_{M\overline \otimes L(K) }((1\otimes v_{g_i})\Delta^\pi(x) (1\otimes v_{h_i}) )&=E_{M\overline \otimes L(K) }(( u_{k_i^{-1}}\otimes 1)  \Delta^\pi( u_{k_i} x u_{l_i}) ( u_{l^{-1}_i} \otimes 1 ))\\ 
&=( u_{k_i^{-1}}\otimes 1) E_{M\overline \otimes L(K) }  (\Delta^\pi( u_{k_i} x u_{l_i}) ) (u_{l^{-1}_i} \otimes 1 )\\
&=( u_{k_i^{-1}}\otimes 1) \Delta^\pi ( E_{Q \rtimes_{\sigma,\alpha} \pi^{-1}(K)}( u_{k_i} x u_{l_i}))  ( u_{l^{-1}_i} \otimes 1 ).\end{split} \end{equation*}

\noindent In particular, this shows that $\|E_{M\overline{ \otimes} L(K) }((1\otimes v_{g_i})\Delta^\pi(x) (1\otimes v_{h_i}) )\|_2= 
\|  E_{Q \rtimes_{\sigma,\alpha} \pi^{-1}(K)}( u_{k_i} x u_{l_i})\|_2$ and in combination with \eqref{weakmixineq} we obtain  $\sum^n_{i=1}\|  E_{Q \rtimes_{\sigma,\alpha} \pi^{-1}(K)}( u_{k_i} x u_{l_i})\|_2\geq C$, for all $x\in \mathscr U(P)$. Then Theorem \ref{corner} yields $P \prec_M Q\rtimes_{\sigma, \alpha} \pi^{-1}(K)$, as desired. 
\end{proof}

In preparation for our second preliminary result we recall a deep result of Popa-Vaes from \cite{PV12} regarding the classification of normalizers of amenable subalgebras in various crossed products von Neumann algebras. Since we need this theorem only for tensor products we will state it in this form.   
\begin{theorem}[\text{\cite[Theorem 1.4]{PV12}}]\label{PV12} Let $\Gamma$ be a group that is biexact \cite{BO08} and weakly amenable \cite{Oz08}. Let $P$ be a tracial von Neumann algebra and denote by $M= P\overline \otimes L(\Gamma)$. Let $q\in M$ be a projection and let $Q\subset qMq$ be a von Neumann subalgebra that is amenable relative to $P$ inside $M$.  Then one of the following must hold: \begin{enumerate}
    \item $Q \prec_M P$;
    \item $\mathscr N_{qMq}(Q)''$ is amenable relative to $P$ inside $M$.
\end{enumerate}
    
\end{theorem}
We note that when $\Gamma$ is amenable then for \emph{any} von 
 Neumann  subalgebras $Q\subseteq qMq$ we automatically have that both $Q$ and $\mathscr N_{qMq}(Q)''$ are amenable relative to $P$ inside $M$, \cite[Proposition 2.4]{OP07}. Thus in this case the result does not provide any meaningful information, and thus when using this result we are  interested only in the case when $\Gamma$ is non-amenable. Theorem \ref{PV12} covers a variety of groups $\Gamma$ which are very important to the structural study of von Neumann algebras. Next we highlight only one such class that is relevant for our results.
Ozawa established that all hyperbolic groups are bi-exact \cite{Oz03,BO08} and weakly amenable \cite{Oz08}. Since both these properties are hereditary, they are satisfied by all subgroups of hyperbolic groups. Using the strong Tits alternative for hyperbolic groups \cite{Gr87} it follows that every amenable subgroup of a hyperbolic group is  elementary\footnote{A group is called elementary if it contains a cyclic subgroup of finite index.}; thus, Theorem \ref{PV12} is meaningful and applies to all non-elementary subgroups of a given hyperbolic group.

Our second preliminary result is a straightforward extension of \cite[Theorem 6.15]{CIOS21}. The proof is almost identical with the one presented there.

\begin{theorem}\emph{\cite{CIOS21}}\label{comutcontrol}
Let $G$ be a non-elementary subgroup of a hyperbolic group, and let $G\curvearrowright^{\sigma,\alpha} (Q,\tau)$ be a trace preserving cocycle action on a tracial von Neumann algebra  $(Q,\tau)$. Let $M=Q\rtimes_{\sigma,\alpha} G$ and let $p\in M$ be a projection. Then the following hold:

\begin{enumerate}
\item [1.] Let $P\subset p M p$ be a von Neumann subalgebra which is amenable relative to $Q$ inside $M$, and let $N=\mathscr N_{pM p}(P)''$. If there is a von Neumann subalgebra $S\subseteq N$ with the relative property (T) such that $S\nprec_{M}Q$, then  $P\prec_{M}^{\rm s}Q$.
\item [2.] Let  $A,B \subset p M p$ be commuting von Neumann subalgebras  and let $N=\mathscr N_{pM p}(A\vee B)''$. If  there is a von Neumann algebra $S\subseteq  N$ with the relative property (T) and such that $S\nprec_{M} Q$, then $A\prec_{M} Q$ or $B\prec_{M} Q$. 
\end{enumerate}
\end{theorem}

\begin{proof}1.\ Throughout the proof we use freely the Notation \ref{com}. Assuming the conclusion is false, by Lemma \ref{lemma.elemetary.facts.intertwinining}(1) we find a nonzero projection $z\in N'\cap pM p$ such that $P z\nprec_{M}Q$. Since $P$ is amenable relative to $Q$ inside $M$, we get that $\Delta(P)$ is amenable relative to $M\overline{\otimes}1$ inside $\tilde M$. This follows from $\Delta(Q)=Q\otimes 1\subset M\otimes 1$ and $\Delta(P)$ being amenable relative to $\Delta(Q)$ inside $\tilde M$ (see also \cite[Proposition 2.4(3)]{OP07}).
Since $G$ is  a non-elementary subgroup of a hyperbolic group, applying Theorem \ref{PV12} in the special case of tensor product to $\Delta(P z)\subset M\overline{\otimes}L(G)$ gives that either a) $\Delta(P z)\prec_{M\overline{\otimes}L(G)}M\overline{\otimes} 1$ or b) $\Delta(N z)$ is amenable relative to $M\overline{\otimes} 1$ inside $M\overline{\otimes}L(G)$.
\vskip 0.03in
If a) holds, then by Proposition \ref{CIK13} we have that $P z\prec_{M}Q$, which is a contradiction. 
If b) holds, then there is a sequence $\eta_n\in L^2(\Delta(z)\langle M\overline{\otimes}L(G),M\overline{\otimes}1\rangle\Delta(z))^{\bigoplus\infty}$ such that $\|\langle \cdot\eta_n,\eta_n\rangle-\tau(\cdot)\|\rightarrow 0$, $\|\langle\eta_n\cdot,\eta_n\rangle-\tau(\cdot)\|\rightarrow 0$, and $\|y\eta_n-\eta_n y\|_2\rightarrow 0$, for every $y\in\Delta(N z)$ (see \cite[Remark 7.1]{CIOS21}).
 Since  $S\subset N$ has the relative property (T),   by \cite[Proposition 4.7]{Po01b},  so does $\Delta(S z)\subset\Delta(N z)$. Hence, there is a nonzero $\eta\in L^2(\Delta(z)\langle M\overline{\otimes}L(G),M\overline{\otimes} 1\rangle\Delta(z))$ such that $y\eta=\eta y$, for every $y\in\Delta(S z)$. 
Then $\zeta=\eta^*\eta\in  L^1(\Delta(z)\langle M\overline{\otimes}L(G),M\overline{\otimes} 1\rangle\Delta(z))$ is nonzero and satisfies $\zeta\geq 0$ and $y\zeta=\zeta y$, for every $y\in\Delta(S z)$. Let $t>0$ such the spectral projection $a={\bf 1}_{[t,\infty)}(\zeta)$ of $\zeta$ is nonzero. Then $a\in \Delta(S z)'\cap \Delta(z)\langle M\overline{\otimes}L(G),M\overline{\otimes} 1\rangle\Delta(z)$. As $ta\leq\zeta$, we get that $\text{Tr}(a)\leq{\text{Tr}(\zeta)}/{t}<\infty$.  Theorem \ref{corner} implies that $\Delta(S z)\prec_{M\overline{\otimes}L(G)}M\overline{\otimes} 1$.  Applying Proposition \ref{CIK13} again, we get that $Sz\prec_{M}Q$ and hence $S\prec_{M}Q$, a contradiction.
\vskip 0.03in
2. 
Let $X\subset \Delta(A)$ be an arbitrary amenable von Neumann subalgebra. Since $X$ and $\Delta(B)$ commute, by Theorem \ref{PV12} we have that either  a) $X \prec_{M\overline{\otimes}L(G)} M\overline{\otimes} 1$ or b) $ \Delta(B)$ is amenable relative to $M\overline{\otimes} 1$ inside $M\overline\otimes L(G)$. If a) holds for all such $X$, then \cite[Corollary F.14]{BO08}  implies that c) $\Delta(A)\prec_{M\overline{\otimes}L(G)} M\overline{\otimes} 1$. If b) holds, by applying Theorem \ref{PV12} we get that either d) $\Delta(B) \prec_{M\overline\otimes L(G)}M\overline{\otimes} 1$  or e) $\Delta (A\vee B)$ is amenable relative to $M \overline{\otimes} 1$ inside $M\overline{\otimes} L(G)$. If e) holds then applying Theorem \ref{PV12} again we further obtain either f) $\Delta(A\vee B)\prec_{M\overline{\otimes}L(G)} M\overline{\otimes} 1$ or g) $\Delta (N)$ is amenable relative to $M \overline{\otimes} 1$ inside $M\overline{\otimes} L(G)$. By Proposition \ref{CIK13}, d) implies that $B\prec_{M} Q$ while c) and f) imply $A\prec_{M} Q$. So it remains to treat g). If g) holds, then 
by arguing as in part 1.\ we get that $S\prec_{M}Q$, which is a contradiction.\end{proof}
\vskip 0.08in
For our results we need a strengthening of the notion of primeness for von Neumann algebras.

\begin{definition}\label{def:vprime} A von Neumann algebra $M$ is called \emph{s-prime} if it satisfies the following property: given any non-zero projection $ p\in M$, if $A, B\subseteq pMp$ are commuting von Neumann subalgebras such that $A\vee B \subseteq pMp$ has finite index, then either $A$ or $B$ is finite dimensional.\end{definition}
From the definition, it is clear that s-primeness implies primeness. The converse however, does not hold true in general. Constructing such examples is typically involved technically and exceeds the purpose of this paper. However, to give the reader a flavor on the presentation of such factors, we highlight an example below that arises from \cite{DHI16} (see also \cite{CD19}). Consider any action by outer automorphisms  $K\curvearrowright^\rho \mathbb F_2$ of a finite group $K$ on the free group with two generators, $\mathbb F_2$. Now let $K \curvearrowright^{\tilde \rho}  \mathbb F_2 \times \mathbb F_2$ be the canonical diagonal action and let $K= (\mathbb F_2 \times \mathbb F_2 )\rtimes_{\tilde \rho} K$ be the corresponding semidirect product group. \cite[Theorem C]{DHI16} yields that $L(K)$ is a prime factor. However, one can easily see that $L(K)$ is not s-prime as it contains the finite index non-prime subfactor, $L(\mathbb  F_2\times \mathbb F_2)= L(\mathbb F_2)\overline \otimes L(\mathbb F_2)$. Additional examples of prime factors associated with fibered products and other groups have been constructed in \cite{DHI16,CD19,Dr22}.

Using our preliminary results together with standard technology on wreath product von Neumann algebras from \cite{Po03,Io06,IPV10} we now derive the main primeness result of the section.
\begin{theorem}\label{primewlp} Let $Q$ be any non-trivial tracial von Neumann algebra and let $\Gamma$ be any icc subgroup of a hyperbolic group. Then any property (T) wreath-like product factor $ M\in \mathcal W\mathcal R(Q,\Gamma)$ is s-prime.    
\end{theorem}
\begin{proof}  Fix a projection $0\neq p \in M$ and two commuting von Neumann subalgebras $A, B\subseteq pMp$ such that $[pMp: A\vee B ]<\infty$.  Since $M$ is a factor, we can apply Lemma \ref{Theorem.findex}(3) to obtain that $\mathcal Z(A\vee B)$ is completely atomic. Hence, we may assume that $A\vee B$ is a factor, up to replacing $p$ by a smaller projection. This implies that both $A$ and $B$ are factors, and thus $A\vee B=A\overline{\otimes}B$.
Since $M$ has property (T), so does $pMp$. Since $[pMp:A\vee B]<\infty$, then using \cite[Proposition 5.7.1) and Proposition 4.6.1)]{Po01b}, we get that both $A$
and $B$ have property (T).  

Now assume by contradiction that neither $A$ nor $B$ are finite dimensional. This implies that both $A$ and $B$ are diffuse von Neumann algebras.  

Denote $P=\overline{\otimes}_\Gamma Q\subset M$. Notice that since $[pMp: A\vee B]<\infty$ and $\Gamma $ is infinite then using Lemma \ref{findexint} and part 5) in Lemma \ref{lemma.elemetary.facts.intertwinining} we have that $A\vee B \nprec_M P $. Thus   using part 2) in Theorem \ref{comutcontrol} (for $R= A \vee B$)  we have either $A\prec_M P$ or $B\prec_M P$. Due to symmetry we may assume that $A\prec_M P$. 

Next we argue there exists a finite subset $F\subset \Gamma$  such that \begin{equation}\label{intfinite}A\prec_M \overline{\otimes}_F Q.\end{equation}

As $A\prec_M P$ one can find nonzero projections $a_0\in A$, $b_0\in P$, a partial isometry $w\in b_0Ma_0$, and a $\ast$-isomorphism onto its image $\psi: a_0 A a_0 \ra \psi(a_0Aa_0)=: P_0\subseteq b_0 Pb_0$ such that \begin{equation}\label{inter7}\psi(x)w=wx\text{ for all }x\in a_0Aa_0. \end{equation} Next we show that one can find a finite subset $F\subset \Gamma$  such that \begin{equation}\label{inter8}P_0\prec_P \overline{\otimes}_F Q.\end{equation} Towards this note that since $A$ has property (T) then so does $a_0Aa_0$. Therefore $P_0$ has property (T) as well. Now fix an exhaustion $F_n \nearrow \Gamma$  by a sequence of increasing finite subsets, where $n\geq 1$. Thus, if $P_n:= \overline \otimes_{F_n} Q$ we notice that $(P_n)_n \subset P$ forms an increasing sequence of von Neumann algebras such that $\overline{\cup_n P_n}^{\rm wot}=P$. Hence, letting $E_{P_n}$ be the conditional expectation from $P$ onto $P_n$ we have that $E_{P_n} \rightarrow {\rm Id}$, $\|\cdot \|_2$-pointwise on $P$. As $P_0$ has property (T) it follows that for every $\varepsilon >0$ there is $n_\varepsilon \in \mathbb N$ such that $\|E_{P_n}(x)-x\|_2< \varepsilon$ for all $x\in \mathscr U(P_0)$ and $n\geq n_\varepsilon$. Picking $\varepsilon\leq  2^{-1}\tau(b_0)^{1/2}$, and using the triangle inequality above we have  $\|E_{P_n}(x)\|_2\geq \|x\|_2 -\|x-E_{P_n}(x)\|_2\geq \|x\|_2-\varepsilon\geq 2^{-1} \tau(b_0)^{1/2} >0$, for all $x\in \mathscr U(P_0)$. Then using Theorem \ref{corner} we get the intertwining \eqref{inter8}, as desired.

Finally, the intertwinings \ref{inter8} and \ref{inter7} together with the transitivity property from \cite[Lemma 1.4.5]{IPP05} (see also \cite[Remark 3.8]{Va08}) yield the intertwining \eqref{intfinite}.

Consequently, letting  $S= \overline \otimes_F Q$,  one can find nonzero  projections $a\in A$, $r\in S$, a nonzero partial isometry $v\in rM a$ and a $\ast$-isomorphism onto its image $\phi: aAa \rightarrow \phi(aAa)=:D \subseteq rSr$ such that \begin{equation}\label{intrel}
    \phi(x)v=vx \text{ for all } x\in aAa.
\end{equation}  
The intertwining relation also implies that $vv^*\in D'\cap r Mr$  and $v^*v\in (A'\cap pMp) a$.

Next we prove the following: 

\begin{claim}\label{controlcom}There exists a finite set $K \subset \Gamma$ such that $D'\cap rMr\subseteq PK$.  
\end{claim}
\emph{Proof of the claim.} Fix $y\in D'\cap rMr$ with $\| y\|_\infty \leq 1$. Let $y =\sum_g y_gu_g$ be its Fourier expansion, where $y_g\in r P$  for all $g\in \Gamma$. Since $xy =yx$ for all $x\in D\subset P$ we get that \begin{equation}\label{intg}x y_g = y_g \sigma_g (x)\text{ for all }x\in D, g\in \Gamma.\end{equation} Let $K= F^{-1}F$ and notice that for every $g \in \Gamma \setminus K$ we have that $gF \cap F =\emptyset$. 

Next, we prove that $y_g=0$ for all $g\in \Gamma \setminus K$ which will conclude the proof of the claim. Fix $\varepsilon>0$. Since the Fourier coefficients satisfy $\|y_g\|_\infty\leq \|y\|_\infty\leq 1$ one can pick $a_1,...,a_k \in rS$ and $b_1,...,b_k \in \overline{\otimes}_{\Gamma\setminus F} Q$ such that \begin{eqnarray}&&\label{ineqnorm}\|\sum^k_{i=1} a_i\otimes b_i\|_\infty\leq 1, \text{ and}\\ &&\label{ineqnorm2}\| y_g -\sum^k_{i=1} a_i \otimes b_i\|_2\leq \varepsilon.\end{eqnarray} Using the intertwining relation \eqref{intg} then  inequalities \eqref{ineqnorm2}, \eqref{ineqnorm}, and \eqref{ineqnorm2} again,  basic calculations show that for every unitary $x\in D$ we have 
\begin{equation}\label{zeroexp}\begin{split}
   \| E_{S}(y_g y_g^*)\|_2 &=  \|x E_{S}(y_g y_g^*)\|_2=\|E_{S}(y_g \sigma_g (x) y_g^*)\|_2 \\
    &\leq \varepsilon +  \| E_{S}((\sum^k _{i=1}a_i\otimes b_i) \sigma_g (x) y^*_g)\|_2 \\ 
    &\leq 2\varepsilon +  \| E_{S}((\sum^k _{i=1}a_i\otimes b_i) \sigma_g (x) (\sum^k _{j=1}a^*_j\otimes b^*_j))\|_2 \\
     &= 2\varepsilon +  \| E_{S}(\sum^k _{i,j=1}(a_i a_j^*) \otimes (b_i\sigma_g(x) b^*_j))\|_2. \end{split}
\end{equation}
Using that $a_i a_j^*\in S$, the $S$-bimodularity of the expectation, $E_S(x)=\tau(x)1$ for all $x\in \overline{\otimes}_{\Gamma \setminus K} Q$, and the triangle inequality  we can also see that 
  \begin{equation}\label{zeroexp2}\begin{split}\| E_{S}(\sum^k _{i,j=1}(a_i a_j^*) \otimes (b_i\sigma_g(x) b^*_j))\|_2  & =\|\sum^k _{i,j=1} (a_i a^*_j) \otimes E_{S}(b_i \sigma_g (x) b^*_j)\|_2 \\ &=\|\sum^k _{i,j=1} \tau( \sigma_g (x) b^*_jb_i)  a_i a^*_j \|_2\\
  &\leq\sum^k _{i,j=1} |\tau( x \sigma_{g^{-1}} (b^*_jb_i))| \| a_i a^*_j \|_2.\end{split}\end{equation}

Since $A$ is diffuse, so is  $D$, and thus one can find a sequence $(x_n)_n\subset \mathscr U(D)$ such that $x_n \rightarrow 0$ weakly, as $n\rightarrow \infty$. In particular, we have that $\lim_{n\rightarrow \infty}|\tau( x_n \sigma_{g^{-1}} (b^*_jb_i))|=0$, for all $i,j\in \overline {1,k}$. Therefore, applying inequalities \eqref{zeroexp} and \eqref{zeroexp2} to $x=x_n$ and taking the limit as $n\rightarrow \infty$ we get that $\| E_{S}(y_g y_g^*)\|_2\leq 2\varepsilon$. Since $\varepsilon>0$ was arbitrary we get $E_{S}(y_g y_g^*)=0$ and hence $y_g=0$, as desired. $\hfill\blacksquare$
\vskip 0.07in
Now, fix  $z\in A'\cap pMp$. Using relation \eqref{intrel} twice we can see that for all $x\in aAa$ we have  $\phi(x) vzv^* = vx zv^*= vzx v^*= vz v^*\phi (x)$ and therefore  $vzv^*\in D'\cap r Mr$.  Thus $v (A'\cap pMp) v^*\subseteq D'\cap rMr$. From \eqref{intrel} we also have that $vAv^*=Dvv^*$ where $vv^*\in D'\cap rMr$. Combining these with the Claim \ref{controlcom} we can see that  
\begin{equation*}
    v ((A) (A'\cap pMp))v^*= v A v^* v (A'\cap p Mp) v^*\subset D (D'\cap rMr) \subseteq PK.
\end{equation*}
This implies that \begin{equation}\label{contain}
 v (A \vee (A'\cap pMp) )v^*\subseteq PK.   
\end{equation} Let $\mathcal P_{PK}$ be the orthogonal projection from $L^2(M)$ onto the $\|\cdot\|_2$-closure of ${\rm span}\{ xu_g \,:\, x\in P, g\in K\}$. One can check for every $x\in M$ we have \begin{equation}\label{projformula}
    \mathcal P_{PK}(x)= \sum_{g\in K} E_{P}(xu_{g^{-1}})u_g.
\end{equation} By \eqref{contain}, for all $x\in v (A \vee (A'\cap pMp) )v^*$ we have that $\mathcal P_{PK}(x)=x$. Using this together with formula \eqref{projformula} and basic calculations,  for all unitaries $u\in v^*v(A \vee (A'\cap pMp) v^*v$ we have 

\begin{equation*}\begin{split}
    0<\tau(v^*v)=\|u\|^2_2=\|v u v^*\|^2_2= \| P_{PK}(x)\|_2^2= \sum_{g\in K}\| E_P(vu v^*u_{g^{-1}})\|_2^2.
\end{split}
    \end{equation*}

Using Theorem \ref{corner}(3), this implies that $A\vee (A'\cap pMp) \prec_M P$. By Lemma \ref{lemma.elemetary.facts.intertwinining}(2) there is a non-zero projection $z\in \mathcal Z (A'\cap pMp)$ such that 
\begin{equation}\label{eq11}
    (A\vee (A'\cap pMp))z\prec_M^s P.
\end{equation}
Using that $A \vee B \subseteq pMp$ has finite index and $A \vee B \subseteq A \vee (A'\cap pMp)$, it follows that $A\vee (A'\cap pMp) \subseteq pMp$ also has finite index.
By Lemma \ref{Theorem.PP86}(2) we deduce that $pMp\prec_{pMp} (A\vee (A'\cap pMp))z$. In combination with \eqref{eq11} we can apply Lemma \ref{lemma.elemetary.facts.intertwinining}(4) and deduce that
 $pMp\prec_M P$, which by Lemma \ref{findexint} entails that $\Gamma$ is finite, a contradiction.    
\end{proof}

\section{A Class of Existentially Closed II$_1$ Factors}\label{Section.ec.factors}

 A group $G$ is called existentially closed if every finite set of equations and inequalities defined over $G$ and is soluble over $G$ is actually soluble in $G$. Existentially closed groups actually coincide with the class of non-trivial algebraically closed groups. These groups display remarkable structural properties and have been intensively studied over the years; the reader may consult  \cite{HS88} for a comprehensive account on this direction.  

More recently, in the field of continuous model theory, Farah-Goldbring-Hart-Sherman \cite{FGHS16} introduced a natural von Neumann algebraic counterpart of these objects.  Roughly speaking, a von Neumann algebra $M$ is called existentially closed if whenever a system of equations with coefficients in $M$ has a solution in an extension of $M$, then it has an approximate solution in $M$.  In more rigouros terms, let $M \subseteq N$ be an inclusion of separable tracial von Neumann algebras.  Following \cite{FGHS16}, one says that $M$ is \emph{existentially closed in $N$} if there exists an embedding $j: N \hookrightarrow M^\omega$, whose restriction on $N$ is the diagonal embedding $M \subset M^\omega$. A separable tracial von Neumann algebra $M$ is called \emph{existentially closed} 
if it is existentially closed in any separable extension, $M\subseteq N$. For the remaining part of the paper we assume the Continuum Hypothesis. 

Next, we recall several fundamental facts of existentially closed von Neumann algebras that will be used freely throughout this paper. The first three are well-known and were proved in \cite{GHS13,FGHS16}, the fourth item follows directly from definitions, while the last has been established in \cite{Go18}. 

\begin{theorem}\label{theorem.model.theory} \emph{\cite{GHS13,FGHS16,Go18}}\label{propecfactors} The following hold true:
\begin{enumerate}
    \item Every existentially closed tracial von Neumann algebra is a II$_1$ factor satisfying McDuff's property;
    \item Every separable tracial von  Neumann algebra embeds into a separable existentially closed  factor;
    \item Every separable tracial von Neumann algebra embeds into any ultrapower $M^\omega$ of any existentially closed factor $M$;
    \item If $M_1\subseteq M_2 \subseteq \cdots \subseteq  M_n \subseteq \cdots$ is a chain of existentially closed factors, then their union $M= \overline {\cup_n M_n}^{\rm wot}$ is also existentially closed;
    \item If $M$ is an existentially closed factor, then for any property (T) subfactor $N\subset M$, its double coummutant satisfies $(N'\cap M)'\cap M= N$.
\end{enumerate}

\end{theorem}

Throughout the paper we will denote by $\mathcal E$ the class of all existentially closed factors. 

Next, we highlight a family of existentially closed \ factors which in spirit resembles the inductive limit factors considered in \cite{Po09} and occurs naturally by combining the aforementioned model theoretic properties with the more recent von Neumann algebraic techniques from \cite{CDI22}. In fact, a very similar construction has been already considered in \cite{CDI22} and various structural results have been obtained there; the interested reader may consult \cite[Section 6]{CDI22}.  
\vskip 0.05in
Henceforth we will denote by $\mathcal {ET}$  the class of all existentially closed  factors that can be presented as inductive limits \begin{equation*}Q=\overline{\bigcup_{n\in \mathbb N} N_n}^{\rm wot},\end{equation*} where $(N_n)_{n\in \mathbb N}$ is an increasing sequence of property (T) s-prime II$_1$ factors satisfying that for every $n$ there is $r>n$ such that $N_n\subset N_r$ has infinite Jones index. 

This family of factors is quite vast. In fact, recycling the interlacing argument from the proof of \cite[Proposition 8.1]{CDI22} we show the following.

\begin{theorem}\label{ectensorindec1}\label{theorem.class.et.embeddable}
The class $\mathcal {ET}$ is {\it embedding universal}, i.e., every separable II$_1$ factor embeds into an element of $\mathcal{ET}$. In particular, $\mathcal {ET}$ is uncountable.  
\end{theorem}

\begin{proof} Fix $P$ a separable von Neumann algebra. Then pick $Q_1\in \mathcal E $ such that $P\subset Q_1$. Using Theorem \ref{theorem.cdi}(2) and Theorem \ref{primewlp} one can find a property (T), virtually prime II$_1$ factor $N_1$ such that   $Q_1\subset N_1$. Since $\mathcal E$ is embedding universal there exists a separable  II$_1$ factor $Q_2\in \mathcal E$ such that $N_1 \subset Q_2$. Since $Q_2$ is existentially closed then by part (1) in Theorem \ref{theorem.model.theory} it is McDuff; in particular, they do not have property (T). Since property (T) passes to finite index suprafactors and $N$ has property (T) it follows that the inclusion $N_1\subset Q_2$  has infinite index.  Continuing on this fashion, by  induction, one can find an increasing  sequence $(N_n)_{n\in \mathbb N} $ of property (T), virtually prime II$_1$ factors and an increasing sequence $(Q_n)_{n\in \mathbb N}$ of separable existentially closed II$_1$ factors satisfying    
\begin{equation}\label{sequence}
    P\subset Q_1\subset N_1\subset Q_2\subset N_2\subset \cdots \subset Q_n\subset N_n\subset \cdots 
\end{equation}
Let $M$ be the inductive limit II$_1$ factor arising from the sequence \eqref{sequence}. By construction we have $M= \overline {\cup_n N_n}^{\text{wot}}= \overline{\cup_n Q_n}^{\text{wot}}$. Since $Q_n\in \mathcal E$ for all $n\in \mathbb N$, then using part 4) in Theorem \ref{propecfactors} we get that $M\in \mathcal E$. Since in our construction $N_n \subset Q_{n+1}$ has infinite index then so is $N_n\subset N_{n+1}$, for all $n\in \mathbb N$. 

Thus $\mathcal {ET}$ is embedding universal and using \cite[Corollary 3]{Oz02} we conclude that  $\mathcal {ET}$ is uncountable. 
\end{proof}
\vskip 0.03in

We conclude this section by recalling the class of infinitely generic II$_1$ factors introduced in  \cite[Propositions 5.7, 5.10 and 5.14]{FGHS16} (see also  \cite[Fact 6.3.14]{AGKE20}). 

\begin{proposition}\emph{\cite{FGHS16}}\label{infgen}
There is a class of separable II$_1$ factors $\mathcal G$ satisfying the following:
\begin{enumerate}
    \item $\mathcal G$ is embedding universal,
    \item  any embedding $\pi: Q_1 \hookrightarrow Q_2$, for some $Q_1,Q_2\in  \mathcal G$ is elementary, i.e., it extends to an isomorphism $Q_1^\omega\cong Q_2^\omega$, and
    \item $\mathcal G$ is the maximum class with properties (1) and (2).
\end{enumerate}
 The elements of $\mathcal G$ are called \emph{infinitely generic II$_1$ factors}. 
\end{proposition}

\begin{remark} The proof of Theorem \ref{theorem.class.et.embeddable} can be used to construct inductive limits of property (T) factors within any family $\mathcal G$ of existentially closed  factors that is both embedding universal and closed under inductive limits. For example, by \cite{FGHS16}, this applies when $\mathcal G$ is the family of all infinitely generic factors. \end{remark}

\section{Tensor Indecomposability Results for Factors in Class $\mathcal {ET}$}\label{Section.ec.factors.indecomposability}

It is well-known that existentially closed groups display strong structural properties, e.g., they are simple \cite[Theorem 1.8]{HS88}. Therefore, they do not admit any nontrivial direct or semi-direct product decompositions. Since existentially closed factors are von Neuman algebraic counterparts of these objects it is reasonable to expect that they share similar indecomposability properties. To stimulate the development of new technology to tackle such properties, it would be natural to first understand if there are certain types of tensor decompositions that existentially closed factors cannot have.

In this direction, by combining \cite[Corollary 2.6]{IS18} with a result from \cite{AGKE20} (see for instance \cite[Proposition 6.2.11 and Proposition 6.3.2]{AGKE20}), one can see that existentially closed \ factors $P$ do not admit any diffuse tensor decompositions of the form $P = R \overline \otimes B$, where $R$ is the hyperfinite II$_1$ factor and $B$ is an arbitrary full II$_1$ factor; in other words, $P$ is not strongly McDuff. 
However, besides this tensor indecomposability result, little is known regarding the possible tensor decompositions of existentially closed factors.

\vskip 0.05in 

In this section we make new progress on this problem by showing that the existentially closed \ factors in the class $\mathcal {ET}$ satisfy an even stronger tensor product indecomposability statement. Specifically, using a m\'elange of methods which combines spectral gap arguments and various intertwining techniques from \cite{Po02,Po09,Io11,CIK13,CKP14,Go18,CIOS21,CDI22} with the primeness result from Theorem \ref{primewlp} we obtain the following.  

\begin{theorem}\label{etprime}\label{theorem.class.et.prime} For any $P\in \mathcal {ET}$, we have that $P\ncong A\overline \otimes B$, for all II$_1$ factors $A$ and $B$, with $B$ full. 

\end{theorem}

\begin{proof} Suppose that $P=\overline{\cup_n N_n}^{\rm wot}$ where $(N_n)_n$ is an increasing sequence of property (T), virtually prime II$_1$ factors satisfying that for every $n$ there is $r>n$ such that $N_n\subset N_r$
has infinite index.

Now assume by contradiction that $P= A\overline \otimes B$ is a decomposition into II$_1$ factors with $B$ full. By applying \cite[Proposition 6.2.11]{AGKE20} we get that $A$ is non-amenable.  

As $B$ is a full factor, by \cite{Co76} we can find $b_1,...,b_k\in \mathscr U(B)$ and $C>0$ such that for every $x\in P$ we have 
\begin{equation}\label{sgap1}
    \sum_{i=1}^k\|xb_i-b_i x\|_2\geq C \|x-E_A(x)\|_2.
\end{equation}
Fix $0<\varepsilon<1$. Since $b_i\in P=\overline{\cup_n N_n}^{\rm wot}$ one can find $n\in \mathbb N$ and $p_1,...,p_k \in N_n$ such that \begin{equation}\label{approx}\|b_i-p_i\|_2\leq \frac{\varepsilon C}{k}\text{ for all }i=\overline{1,k}.  \end{equation}
Fix $y\in (N_n'\cap P)_1$. Using inequalities \eqref{sgap1}, \eqref{approx} along with relations $yp_i=p_iy$ for all $i=\overline{1,k}$ we get \begin{equation*}\begin{split}
  \|y-E_A(y)\|_2\leq  \frac{1}{C} \sum_{i=1}^k\| yb_i-b_i y\|_2 \leq    \varepsilon. 
\end{split}\end{equation*} 

Using triangle inequality this further implies for all unitaries $y\in N_n'\cap P$ we have $\|E_A(y)\|_2\geq 1-\varepsilon>0$. Thus part (3) in Theorem \ref{corner} shows that $N_n'\cap P \prec_M A$. 

Since $N_n$ has property (T), it has $w$-spectral gap in the sense of \cite{Po09} in any extension (in fact this characterizes property (T), see \cite{Ta22}). As $P$ is existentially closed,  by Theorem \ref{theorem.model.theory}(5) we get that $(N_n'\cap P )'\cap P=N_n$. Hence, by passing to relative commutants using Lemma \ref{lemma.elemetary.facts.intertwinining}(3) that  \begin{equation}\label{sintertwining}B\prec_P N_n.\end{equation}

In the rest of the proof we show that \eqref{sintertwining}  will lead to a contradiction. 
\vskip 0.05in 

 For simplicity, denote $Q:=N_n$.  As $B\prec_P Q$, using \cite[Proposition 2.4]{CKP14} and its proof, one can find projections $b \in B$, $q \in Q$, a nonzero partial
isometry $v\in qPb$
, a von Neumann subalgebra $D \subseteq  q Q q$, and a $\ast$-isomorphism onto its image $\phi : bBb \rightarrow D $ satisfying the following relations:

\begin{enumerate}
    \item $D \vee (D' \cap q Q q) \subseteq  qQ q$ has finite index;
    \vskip 0.04in
    \item $\phi(x)v = vx$ for all  $x \in bBb$;
    \vskip 0.04in
    \item $vv^* \in D' \cap  qPq$ and $p:=v^*v = a\otimes b$ for some projection $a\in A$.
\end{enumerate}

Since $Q$ is s-prime, relation (1) implies that $D'\cap qQq$ is finite dimensional. Fix $0\neq z\in D'\cap qQq$, a minimal central projection. Thus one can find $n \in \mathbb N$ such that $(D'\cap qOq)z\cong M_n(\mathbb C)$. Since $Dz$ is a factor commuting with $(D'\cap qQq)z$ this further implies the index $[(D\vee D'\cap qQq)z : Dz]<\infty$.  Lemma \ref{findex} and relation (1) also imply the index $[zQz: (D\vee (D'\cap qQq))z]<\infty$. Using the transitivity property of finite index inclusions, these relations yield that $Dz \subset z Qz$ is a finite index inclusion of II$_1$ factors. Moreover, replacing $q$ by $qz$, $D$ by $Dz$, $v$ by $vz$ and $\phi(\cdot) $ by $\phi(\cdot)z$, relations (2) and (3) still hold and furthermore, instead of (1) we actually have

\begin{enumerate}
    \item [(1')]$D \subseteq  qQ q$ is a finite index inclusion of II$_1$ factors.
\end{enumerate}

Choosing $u \in \mathscr U( P
)$ such that $v = u p $, then relation (2) entails that  
\begin{equation}\label{equalcorners1} D vv^* = v B v^* = u( p B p)u^*.\end{equation}

Passing to relative commutants above and using (3), we also have  
\begin{equation}\label{equalcorners2} vv^*(D'\cap qPq)  vv^*  = u(pAp)u^*.\end{equation}
Altogether, relations \eqref{equalcorners1}-\eqref{equalcorners2} show that $vv^*(D\vee (D'\cap qPq))  vv^*  = u pP p u^*$. If $t$ denotes the central support of $vv^*$ in $D\vee (D'\cap qPq)$, this further implies that 
\begin{equation}\label{equalcorners3}
    (D\vee (D'\cap qPq))t  = t P t.
\end{equation}

Using relation (1'), Lemma \ref{lemma.finite index.relative.commutant} and Proposition \ref{findex}(1) we deduce that $(qQq'\cap qPq)t \subseteq (D'\cap qPq)t$ has finite index. Since $Dt$ is a factor and commutes with $(D' \cap qPq)t$, then it is in tensor position with respect to $(D'\cap qPq)t$ and $(qQq'\cap qPq)t$; thus, using Lemma \ref{findex}(4) it follows that $ (D \vee (q Q'q \cap qPq))t \subseteq (D\vee (D'\cap qPq))t =tPt$ has finite index as well. By Lemma \ref{Theorem.PP86} we deduce that  $P\prec_P D \vee (q Qq; \cap qPq)$ and since $D\subseteq qQq$ we further have  $P\prec_P qQq \vee (q Qq' \cap qPq)= q (Q \vee (Q '\cap P) )q$. Since $(Q'\cap P )'\cap P=Q$, we get that 
$Q\vee (Q'\cap P) $ has trivial relative commutant inside $P$.
Hence,  by  \cite[Proposition 2.3]{CD18} we further derive that 
\begin{equation}\label{almostproduct}N_n\vee (N_n'\cap P)\subseteq P\text{ is a finite index inclusion of II}_1 \text{ factors}. \end{equation}

Now, fix $r>n$ such that $N_n\subset N_r$ has infinite index. 
Let $E_{N_r}: P \rightarrow N_r$ be the canonical  conditional expectation and denote by $S= E_{N_r}(N'_n\cap P)''$ the von Neumann algebra generated by the image of $N_n'\cap P$ under the expectation  $E_{N_r}$ inside $N_r$. Using the $N_r$-bimodularity of $E_{N_r}$, we can see that $E_{N_r}(N_n'\cap P) \subset N_n'\cap N_r$. In particular, $S$ commutes with $N_n$.

Denote also $T= E_{N_r}(N_n \vee (N'_n\cap P))''$. Since for every $x \in N_n$ and $y \in N_n'\cap P$ we have that $E_{N_r}(xy)= x E_{N_r} (y)$ one can see that $T = N_n \vee S$.

Finally, notice that \eqref{almostproduct} and Lemma \ref{findexunderexpectation} (for $M=N_n\vee (N_n'\cap P)$ and $N=N_r$) imply that  $N_n \vee S =T\subseteq N_r$ has finite index. 
Since $N_r$ is s-prime it follows that $S$ is finite dimensional. Thus, 
 $N_n\subset N_r$ has finite index, which is a contradiction. \end{proof}



\begin{remark} The above proof applies verbatim to any existentially closed factor of the form $P = \overline{\cup_n N_n}^{\rm wot}$, where $N_n$ are virtually prime factors that do not necessarily have property (T) but satisfy the bicommutant property $(N_n'\cap P)'\cap P=N_n $,  for all $n\in \mathbb N$. It would be interesting to determine if there are such existentially closed factors $P$ when $N_n$ are free products.   
\end{remark}

Next, we record some immediate consequences of Theorem \ref{etprime}. 
First, we observe that while factors in class $\mathcal {ET}$ are inner asymptotically central (see \cite[Definition 6.2.8 and Proposition 6.3.2]{AGKE20}), they cannot be written as infinite tensor products of full factors (see also \cite[Question 6.4.1]{AGKE20}). 
\begin{corollary} For any $P\in \mathcal {ET}$, we have that $P\ncong \overline{\otimes}_{n\in \mathbb N} P_n$ for any infinite collection of $\{P_n\,:\,n\in \mathbb N\}$ of full II$_1$ factors.\end{corollary}

The second application concerns the structure of the central sequence algebra of a certain class of existentially closed factors. A factor $P$ is called \emph{super McDuff} if its central sequence algebra $P'\cap P^\omega$ is a II$_1$ factor. In \cite[Question 6.3.1]{AGKE20} it was asked whether all existentially closed factors are super McDuff.
In \cite[Theorem 6.4]{CDI22} this question was answered positively for all infinitely generic factors. Using Theorem \ref{etprime}, we can moreover show that if $P$ is an infinitely generic factor which is also in class $\mathcal {ET}$, then all of its tensor factors are super McDuff. 

\begin{corollary} Let $P\in \mathcal {ET}$ be an infinitely generic factor. Then for any diffuse tensor decomposition $P= A\overline 
\otimes B$, both factors $A$ and $B $ are super McDuff. 

\end{corollary} 

\begin{proof}Since $P$ is infinitely generic, it follows from \cite[Theorem 6.4]{CDI22} that $P$ is super McDuff.   Let $P= A \overline \otimes B$ be a diffuse tensor decomposition. Using Theorem \ref{etprime} it follows that $A$ and $B$ have property Gamma. Hence, their central sequence algebras $A'\cap A^\omega$ and $B'\cap B^\omega$ are diffuse \cite{MD69b}. In addition, as $P'\cap P^\omega$ is a factor,  \cite[Theorem F]{Ma17} further implies that $A'\cap A^\omega$ and $B'\cap B^\omega$ are also factors, which yields the desired conclusion. \end{proof}

We believe that in fact all existentially closed factors satisfy the statement of Theorem \ref{etprime} and thus conjecture the following:

\begin{conjecture}\label{conj:notensorbyfull} For any $P\in \mathcal {E}$, we have that $P\ncong A\overline \otimes B$, for all  II$_1$ factors $A$ and $B$ with $B$ full. 

\end{conjecture}

At the time of writing, we do not have an approach for this conjecture in its full generality. However, we would like to mention that the conjecture holds true if one assumes in addition that $B$ is a group von Neumann algebra arising from a non-inner amenable group; for the definition, see Exercise 6.20 in Ioana’s article in this volume. 
In fact,  we have the following more general indecomposability result.

\begin{theorem}\label{theorem.inneramenable} If $Q\in \mathcal E$, then $Q\ncong A\rtimes_{\sigma, \alpha}
\Gamma$ for any trace-preserving cocycle action $\Gamma \curvearrowright^{\sigma,\alpha}(A,\tau)$ of a non-inner amenable group $\Gamma$.  \end{theorem}

\begin{proof} Assume by contradiction that $Q= A\rtimes_{\sigma, \alpha}
\Gamma$ for some cocycle action $\Gamma \curvearrowright (A,\tau)$. As $\Gamma$ is non-inner amenable, then using a generalization of Exercise 6.20 in Ioana's article in this volume, one can find $g_1,\ldots,g_k\in \Gamma$ and $C>0$ such that for every $x\in Q$ we have 
\begin{equation}\label{sgap}
    \sum_{i=1}^k\|xu_{g_i}-u_{g_i} x\|_2\geq C \|x-E_A(x)\|_2.
\end{equation}
In particular, this implies that $Q'\cap Q^\omega \subseteq A^\omega$. On the other hand, since $Q$ is existentially closed \ then there exists a unitary $(v_n)_n=v\in Q^\omega$ such that $v Qv^*\subseteq Q'\cap Q^\omega$ (see \cite[Definition 6.2.8 and Proposition 6.3.2]{AGKE20}). Altogether, these relations imply that for every $g\in \Gamma$ there exists a sequence $(c_n(g))_n\in \mathcal U(A)$ such  that  $\lim_{n\rightarrow \omega}\|v_n u_g v_n^*- c_n (g)\|_2=0$. This further shows that   \begin{equation}\label{approxinner1}
    \lim_{n\rightarrow \omega}\|v_n - c_n (g) v_n u_{g^{-1}}\|_2=0.  
\end{equation}  

To this end let $v_n =\sum_g v^n_g u_g$ with $v^n_g\in A$ be its Fourier expansion. Using this relation together with the triangle inequality and the fact that $c_n(g)$ and $\alpha(h,g)$ are unitaries, we can see  that
 \begin{equation}\label{approxinner2}\begin{split}
    \|v_n - c_n (g) v_n u_{g^{-1}}\|^2_2&=\sum _h \| v^n _{hg^{-1}}  - c_n (g) v^n_h \alpha(h,g^{-1})\|^2_2\\
    &\geq \sum _h \left (\| v^n _{hg^{-1}}\|_2  - \|c_n (g) v^n_h \alpha(h,g^{-1})\|_2\right )^2\\
    &= \sum _h \left (\| v^n _{hg^{-1}}\|_2  - \|v^n_h \|_2\right )^2.
    \end{split}\end{equation}
    
Letting $\xi_n(g):=\|v_g^n\|_2$ for all $g\in \Gamma$ one can see that $\xi_n\in \ell^2 \Gamma$ with $\|\xi_n\|_2=1$. Moreover, relations \eqref{approxinner1}-\eqref{approxinner2}  show that $\lim_{n\rightarrow \omega}\|\rho_g(\xi_n)-\xi_n\|_2 =0$, where $\rho$ is the right regular representation of $\Gamma$. This implies that $\Gamma$ is amenable, a contradiction. \end{proof}

We continue with a few basic observations that rule out other particular  types of tensor product decompositions for existentially closed factors. It is possible that some of the results could also shed some light towards a possible approach to Conjecture \ref{conj:notensorbyfull}.  

Using embedding results into von Neumann algebraic HHN extensions, as in the group situation, we show that any existentially closed factor $M$ does not decompose as $M= A \overline \otimes B$ for any factors $A$ and $B$ that contain isomorphic copies of a given non-amenable factor.

\begin{theorem}\label{notensorwithisomparts} If $P\in\mathcal E$, then $P\ncong A \overline\otimes B$ for any factors $A$, $B$ that both contain isomorphic copies of a given non-amenable factor.   

\end{theorem}
\begin{proof} Assume by contradiction that $P= A \overline\otimes B$ such that there exist isomorphic non-amenable subfactors $S\subseteq A$, $T\subseteq B$. Fix a $*$-isomorphism $\phi: S\rightarrow T$.
Following an idea from the proof of \cite[Theorem 4]{Go21},
we consider the HNN extension $Q={\rm HNN}(P, \phi)$ associated to $\phi$ \cite{Ue05} and obtain that $P\subset Q$ is an inclusion of tracial von Neumann algebras together with a unitary $w\in Q$ such that $w\phi (x)w^*=x$, for any $x\in S$ (see also \cite[Section 3]{FV10}). 
Since $P$ is existentially closed and $w\in Q\supset P$, it follows that we can represent $w=(u_n)_n\in P^\omega$, and hence,  \begin{equation}\label{eq:almostinvvect}\lim_{n\to\omega}\|u_n\phi(x)-xu_n \|_2 = 0, \text{ for any }x\in S.\end{equation} Consider the Hilbert space  $L^2(P)$ and endow it with the $S$-bimodular structure given by $x\cdot \xi \cdot y= (x\otimes 1)\xi (1\otimes \phi(y))$ for every $x,y\in S$ and $\xi\in L^2(P)$. 
Then $L^2(P)$ is isomorphic to a sub-bimodule of a multiple of the coarse $S$-bimodule, $\left(L^2(S)\overline \otimes L^2(S)\right)^{\oplus \infty}$.
This follows from $P=A\bar\otimes B$ and the fact that any left (respectively, right) $S$-module is contained in $L^2(S)^{\oplus\infty}$ as a left (respectively, right) $S$-module (see, for instance, \cite[Proposition 8.2.3]{AP22}).
Therefore, using Section 2.4  relation  \eqref{eq:almostinvvect} implies that $S$ is amenable, which is a contradiction.   \end{proof}


A well-known conjecture in the theory of von Neumann algebras predicts
that the following version of von Neumann's problem in group theory holds true: 

\begin{conjecture}\label{conj:vNprob}Every nonamenable II$_1$ factor contains a copy of $L(\mathbb F_2)$.\end{conjecture} A positive answer to this conjecture combined with Theorem \ref{notensorwithisomparts} would rule out the existence of tensor product decomposition into non-amenable factors for all existentially closed factors. Unfortunately,  this conjecture is wide open at this time and very difficult to establish in full generality.  However, since existentially closed  factors are highly rich objects, one expects that they all verify Conjecture \ref{conj:vNprob}.  Such a result would lead to new advances regarding Conjecture \ref{conj:notensorbyfull}. Indeed, we first notice the following elementary result.
\begin{theorem}\label{thm:ecpreservundertensorbyfull}
If $Q\in  \mathcal E$ and $Q=A \overline \otimes B$ with $B$ a full factor, then $A\in \mathcal E$. 
\end{theorem}

\begin{proof} First observe that since $B$ is full and $Q=A \overline \otimes B$ then the spectral gap condition \eqref{sgap1} implies that the inclusion $B\subset Q$ has w-spectral gap, i.e.\ \begin{equation}\label{sg} B'\cap Q^\omega = (B'\cap Q)^\omega.\end{equation} 
Now let $A \subseteq C$ be any extension. Thus $Q=A \overline \otimes B\subseteq C\overline \otimes B$. Since $Q=A \overline \otimes B$ is existentially closed we have that $Q\subseteq C\overline \otimes B \subseteq Q^\omega $. However we have that $C\subseteq B'\cap Q^\omega$ and using \eqref{sg} we get $A\subseteq C\subseteq B'\cap Q^\omega=A^{\omega}$.  
\end{proof}

The recent refutation of the Connes Embedding Conjecture from \cite{JNVWY20} implies in particular that the hyperfinite II$_1$ factor is not existentially closed. Thus  Conjecture \ref{conj:vNprob} could potentially hold true for all existentially closed factors. (We point out the corresponding statement for groups holds true as every group with solvable word problem embeds into every existentially closed group.) If this is the case,  then combining Theorems \ref{notensorwithisomparts} and \ref{thm:ecpreservundertensorbyfull} we would obtain that for any existentially closed factor $P$ we have that $P\ncong A\overline \otimes B$ where $B$ is any full factor containing a copy of $L(\mathbb F_2)$.

\vskip 0.05in
We end this section with one last conjecture on tensor decompositions of existentially closed factors. To this end, we first recall some terminology and provide some context. 

If $P$ is a McDuff factor then $P\cong P \overline \otimes R$, where $R$ is the hyperfinite factor.
We say that $P$ admits only the canonical McDuff decomposition if the following holds:  \emph{for any tensor decomposition $P=Q \overline \otimes R$, $Q$ is isomorphic to $P$.}  

Currently, only a few classes of McDuff factors admitting only the canonical McDuff decomposition are known. These include: \begin{enumerate}
    \item All infinite tensor products $\overline{\otimes}_{n\in\mathbb N} P_n$ of full factors $P_n$ \cite{Po09}; see also \cite[Corollary G]{Ma17} for an alternative  shorter proof.
    \item All McDuff's group factors $L(T_0(\Gamma))$ for any nontrivial icc group $\Gamma$, where $T_0(\Gamma)$ are McDuff's groups introduced in \cite{MD69a}; this is the main result in \cite{CS22}.   
\end{enumerate}  
We believe that existentially closed factors also have this property.

\begin{conjecture}
Any factor $P\in \mathcal E$ admits only the canonical McDuff tensor decomposition. 
\end{conjecture}


\section{Open Problems and Final Comments on the Structure of Existentially Closed Factors}\label{section.open}

In this section we propose several additional open problems regarding the structure of existentially closed factors. 

Much of the technology developed within the framework of Popa's deformation/rigidity theory is tailored for analyzing von Neumann algebras that are built from countable groups and their actions on von Neumann algebras. In order to understand if these methods can be applied to the study of existentially closed factors, one first needs to understand if existentially closed factors arise from countable groups or their actions. The following fundamental question is wide open:

\begin{op}\label{grec} Are there any group II$_1$ factors in the class $\mathcal E$? 

\end{op}

A possible approach for this problem, using model theoretic forcing, has been suggested in Goldbring's article in this volume.

If the answer to Problem \ref{grec} is positive, then we can further ask if it is possible to describe the groups $\Gamma$ satisfying $L(\Gamma)\in \mathcal E$? If this is too much to ask for, could one at least identify some properties such groups enjoy? It would be natural to investigate if these groups share any properties with existentially closed groups. Obviously, since existentially closed factors are McDuff, such groups $\Gamma$ might not be simple but is it at least true that they are infinitely generated? Do they always contain free groups?

In the opposite direction, constructing existentially closed factors that do not arise from groups also seems challenging. 

\begin{op}
Are there any concrete examples of existentially closed factors that are not group factors? 
\end{op}
We believe that most existentially closed factors are not group factors. 
The idea would be to exploit the embeddings group factors (e.g., the comultiplication $\Delta:L(\Gamma)\rightarrow L(\Gamma)\overline{\otimes}L(\Gamma)$) possess. However, thus far, we were not able to construct existentially closed factors which are not group factors. We notice that, even by abstract means, so far we do not know the existence of even a single existentially closed group factor. 
 
We end with another open problem regarding indecomposability properties of  existentially closed factors. Theorem \ref{theorem.inneramenable}  implies that existentially closed  factors do not appear as group measure space von Neumann algebras $L^\infty(X)\rtimes \Gamma$ associated with probability measure preserving actions $\Gamma \curvearrowright (X,\mu)$ on diffuse probability spaces, where $\Gamma$ is a non-inner amenable group. We believe that much more should be true and conjecture the following:

\begin{conjecture} There is no factor $P\in \mathcal E$ which has a Cartan subalgebra. 

\end{conjecture}





\end{document}